\documentclass[10pt,twoside
]{amsart}
\usepackage[english]{babel}
\usepackage[all]{xy}
\usepackage{amssymb,amsmath,bbm,url,xr,a4wide}
\usepackage[mathscr]{euscript}
\usepackage{mathrsfs}
\input cyracc.def
\DeclareFontFamily{U}{russian}{}
\DeclareFontShape{U}{russian}{m}{n}
        { <5><6> wncyr5
        <7><8><9> wncyr7
        <10><10.95><12><14.4><17.28><20.74><24.88> wncyr10 }{}
\DeclareSymbolFont{Russian}{U}{russian}{m}{n}
\DeclareSymbolFontAlphabet{\mathcyr}{Russian}
\makeatletter
\let\@math@cyr\mathcyr
\renewcommand{\mathcyr}[1]{\@math@cyr{\cyracc #1}}
\makeatother

\DeclareFontFamily{OT1}{pzc}{}
\DeclareFontShape{OT1}{pzc}{m}{it}%
             {<-> s * [1,150] pzcmi7t}{}
\DeclareMathAlphabet{\mathpzc}{OT1}{pzc}%
                                 {m}{it}
\input cyracc.def
\DeclareFontFamily{U}{russian}{}
\DeclareFontShape{U}{russian}{m}{n}
        { <5><6> wncyr5
        <7><8><9> wncyr7
        <10><10.95><12><14.4><17.28><20.74><24.88> wncyr10 }{}
\DeclareSymbolFont{Russian}{U}{russian}{m}{n}
\DeclareSymbolFontAlphabet{\mathcyr}{Russian}
\makeatletter
\let\@math@cyr\mathcyr
\renewcommand{\mathcyr}[1]{\@math@cyr{\cyracc #1}}
\makeatother

\newcommand{\NN} {\mathbb N}
\newcommand{\ZZ} {\mathbb Z}
\newcommand{\QQ} {\mathbb Q}
\newcommand{\RR} {\mathbb R}
\newcommand{\CC} {\mathbb C}

\renewcommand{\AA} {\mathbb A}
\newcommand{\PP} {\mathbb P}
\newcommand{\GG}{\mathbb G_m}
\newcommand{\BM}{\mathrm{BM}}

\DeclareMathOperator{\Der}{D}
\DeclareMathOperator{\Comp}{C}
\DeclareMathOperator{\Hom}{Hom}
\DeclareMathOperator{\uHom}{\underline{\mathrm{Hom}}}
\DeclareMathOperator{\Tot}{Tot}
\newcommand{\derL}{\mathbf{L}}
\newcommand{\derR}{\mathbf{R}}

\DeclareMathOperator{\sus}{\Sigma^\infty}
\DeclareMathOperator{\cupp} {\scriptstyle\cup\textstyle}

\newcommand{\spe}{\mathrm{sp}}

\newcommand{\HB}{H_\mathcyr B}
\DeclareMathOperator{\CH}{CH}
\DeclareMathOperator{\ch}{ch}
\DeclareMathOperator{\td}{td}

\DeclareMathOperator{\scup}{\scriptstyle\cup\textstyle}
\DeclareMathOperator{\scap}{\scriptstyle\cap\textstyle}

\newcommand{\smx}[1]{{\mathit S\mspace{-1.mu}m}/#1}
\DeclareMathOperator{\PSh}{PSh}
\DeclareMathOperator{\Sh}{Sh}
\DeclareMathOperator{\Sp}{Sp}
\DeclareMathOperator{\Spr}{Sp^{{\rm ring}}}
\DeclareMathOperator{\Ho}{\mathcal Ho}

\DeclareMathOperator{\smod}{\E_\syn-mod}

\newcommand{\DMte}{\smash{\Der_{\AA^1}^{\rm eff}}}
\newcommand{\DMt}{\smash{\Der_{\AA^1}}}
\newcommand{\DMB}{\smash{DM_\mathcyr B}}

\newcommand{\E}{\mathbb E}
\newcommand{\F}{\mathbb F}
\newcommand{\G}{\mathbb G}

\newcommand{\un}{\mathbbm 1}

\newcommand{\rig}{\mathrm{rig}}
\newcommand{\syn}{\mathrm{syn}}
\newcommand{\dR}{\mathrm{dR}}
\newcommand{\FdR}{\mathrm{FdR}}

\newcommand{\rel}{\mathrm{rel}}
\newcommand{\tube}[2]{{]#1[_{#2}}}

\DeclareMathOperator{\Fisoc}{F-isoc}
\DeclareMathOperator{\Cone}{Cone}
\DeclareMathOperator{\End}{End}
\DeclareMathOperator{\Spm}{Spm}

\DeclareMathOperator{\dlog}{dlog}
\DeclareMathOperator{\id}{Id}
\DeclareMathOperator{\Spec}{Spec}

\DeclareMathOperator{\Pic}{Pic}

\DeclareMathOperator{\tot}{tot}

\DeclareMathOperator{\gr}{gr}

\DeclareMathOperator{\Gdm}{Gdm}
\DeclareMathOperator{\tGdm}{\widetilde{Gdm}}

\DeclareMathOperator{\Ext}{Ext}

\DeclareMathOperator{\Coker}{Coker}

\DeclareMathOperator*{\colim}{colim}
\renewcommand{\Im}{\mathrm{Im}}

\title{The rigid syntomic ring spectrum}

\author{F.~D\'eglise, N.~Mazzari}
\date{\today}

\newtheorem{thm}{Theorem}[subsection]
\newtheorem*{thmi}{Theorem}
\newtheorem{prop}[thm]{Proposition}
\newtheorem{cor}[thm]{Corollary}
\newtheorem{lm}[thm]{Lemma}

\theoremstyle{definition}
\newtheorem{df}[thm]{Definition}
\newtheorem{ex}[thm]{Example}
\newtheorem{num}[thm]{}

\theoremstyle{remark}
\newtheorem{rem}[thm]{Remark}

\numberwithin{equation}{thm}

\begin{document}
\email[F. D\'eglise]{frederic.deglise@ens-lyon.fr}
\email[N. Mazzari]{nicola.mazzari@math.u-bordeaux1.fr}

\begin{abstract}
The aim of this paper is to show that rigid syntomic cohomology -- defined by Besser --
 is representable by a rational ring spectrum in the motivic homotopical sense.
 In fact, extending previous constructions,
 we exhibit a simple representability criterion and we apply it to several 
 cohomologies in order to get our central result.
 This theorem gives new results for rigid syntomic cohomology such as h-descent
 and the compatibility of cycle classes with Gysin morphisms.
 Along the way, we prove that motivic ring spectra induce a complete
 Bloch-Ogus cohomological formalism and even more.
 Finally, following a general motivic homotopical philosophy, we exhibit
 a natural notion of rigid syntomic coefficients.
\end{abstract}

\maketitle

\section*{Introduction}

In the 1980s, Beilinson stated his conjectures relating the special values of $L$-functions and the regulator map of a variety $X$ defined over a number field \cite{Beu:84a,Beu:86a}. The regulator  considered by Beilinson  is a map from the $K$-theory of $X$ with target the Deligne-Beilinson cohomology\footnote{Here we assume  that the weight filtration  is part of the definition. This is not the case in the original definition by Deligne, where only the Hodge filtration was considered. See \cite{Beu:Not86} for a complete discussion.}  with real coefficients
$$
reg:K_{2i-n}(X)^{(i)}\otimes \QQ\rightarrow H_{DB}^n(X,\RR(i)) \ .
$$
One can define  $H_{DB}^n(X,A(i))$ for any subring $A\subset \RR$. For $A=\ZZ$,  Beilinson  proved that $H_{DB}^n(X,\ZZ(i))$ is the absolute Hodge cohomology theory: \textit{i.e.} it computes the group of homomorphisms in the derived category of mixed Hodge structures
\[
H_{DB}^n(X,\ZZ(i))=\Hom_{D^b(MHS)}(\ZZ,R\Gamma_{Hdg}(X)(i)[n])\ ,
\] 
where $R\Gamma_{Hdg}(X)$ is the mixed Hodge complex associated to $X$ whose cohomology is the Betti cohomology of $X$ endowed with its mixed Hodge structure \cite{Beu:Not86}. Further Beilinson conjectured that the higher $K$-theory groups form an absolute cohomology theory, in fact the universal one, called motivic cohomology. This vision is now partly accomplished. We do not have the category of mixed motives, but we can construct a triangulated category playing the role of its derived category. More precisely, Cisinski and D\'eglise proved that for any finite dimensional noetherian scheme $X$ there exists a monoidal triangulated category $\DMB(X)=\DMB(X,\QQ)$ (along with the six operations) such that
\[
\HB^{n,i}(X):=\Hom_{\DMB(S)}(\un_S,\pi_*\un_X (i)[n])\simeq K_{2i-n}(X)^{(i)}\otimes \QQ
\]
when $\pi:X\to S$ is a smooth morphism and $S$ is regular \cite{CD3}.

Now let $K$ be a $p$-adic field (\textit{i.e.} a finite extension of $\QQ_p$) with ring of integers $R$. Given $X$ a smooth and algebraic $R$-scheme, Besser  defined the analogue of the Deligne-Beilinson cohomology in order to study the Beilinson conjectures for $p$-adic $L$-functions \cite{Bes:Syn00}. The work of Besser extends a construction initiated by Gros  \cite{Gros}. The cohomology defined by Besser is called the \emph{rigid syntomic}\footnote{The word {\em rigid} is due to the fact that the rigid cohomology of Berthelot plays a role in the definition. The word {\em syntomic} comes from the work of Fontaine-Messing \cite{FonMes:87a} where the syntomic site was used to define a cohomology theory strictly related to the one of Besser in the smooth and projective case.} cohomology,  denoted by $H^n_\syn(X,i)$. Roughly it is defined as follows: let $R\Gamma_\rig(X_s)$ (resp. $R\Gamma_\dR(X_\eta)$) be a complex of $\QQ_p$-vector spaces whose cohomology is the rigid (resp. de Rham) cohomology of the 
special fiber $X_s$ (resp. generic fiber $X_\eta$) of $X$, then
\[
H^n_\syn(X,i)=H^{n-1}(\Cone(f:R\Gamma_\rig(X_s)\oplus F^iR\Gamma_\dR(X_\eta)\to R\Gamma_\rig(X_s)\oplus R\Gamma_\rig(X_s))\ ,
\]
where $f(x,y)=(x-\phi(x)/p^i, sp(y)-x)$, $\phi$ is the Frobenius map, $sp$ is the Berthelot's specialization map. 

There is a regulator map for this theory and one can also interpret  rigid syntomic cohomology  as an absolute cohomology \cite{Ban:Syn02,ChiCicMaz:Cyc10}.

The aim of the present paper is to represent rigid syntomic cohomology in the triangulated category of motives by a ring object $\E_\syn$. This allows one  to prove that rigid syntomic cohomology is a Bloch-Ogus theory and satisfies h-descent
 (\emph{i.e.} proper and fppf descent). In particular, we obtain that the Gysin map is compatible with the direct image of cycles as conjectured by Besser \cite[Conjecture~4.2]{Bes:12a}. We can say that this paper is the
natural push-out of the work of the first author in collaboration with
Cisinski \cite{CD2} and that of the second author in collaboration with
Chiarellotto and Ciccioni \cite{ChiCicMaz:Cyc10}.\\[1ex]

Let us review in more detail the content of this work. \\
First  we recall some results of the motivic homotopy theory. Let $S$ be a base scheme (noetherian and finite dimensional). To any  object $\E$ in $\DMB(S)$ we can associate a bi-graded cohomology theory
\[
\E^{n,i}(X):=\Hom_{\DMB(S)}(M(X),\E(i)[n])
\]
where $M(X):=\pi_!\pi^!\un_S$ is the (covariant) motive of $\pi:X\to S$.
The cohomology defined by the unit object $\un_S$ of the monoidal category
 $\DMB(S)$ represents rational motivic cohomology denoted by $\HB$.
 When $X$ is regular, $\HB^{n,i}(X)$
 coincides with the original definition of Beilinson
 using Adams operations on rational Quillen K-theory. \\
The category of Beilinson motives $\DMB(S)$ can be constructed using some homotopical machinery starting with the category $\Comp(S,\QQ)$ of complexes of $\QQ$-linear pre-sheaves on the category of affine and smooth $S$-schemes (see \S~1). An object of $\DMB(S)$ should be thought of as a cohomology theory on the category of $S$-schemes which is $\AA^1$-homotopy invariant, satisfies the Nisnevich excision and is oriented (in the sense of remark~\ref{rem:axioms_motivic_ring_sp} point (1)). 

The category of Beilinson motives is monoidal.
Monoids with respect to this tensor structure corresponds
 to cohomology theory equipped with a ring structure.
 Following the general terminology of motivic homotopy theory,
  we call such a monoid a {\em motivic ring spectrum} (Def.~\ref{def:mrs}). 
Given such an object $\E$, the associated cohomology theory $\E^{n,i}(X)$ is
naturally a bi-graded $\QQ$-linear algebra satisfying the following properties:
\begin{enumerate}
\item{\emph{Higher cycle class/regulator}. --}
The unit section of the ring spectrum $\E$ induces a canonical morphism,
 called regulator:
$$
\sigma:\HB^{n,i}(X) \rightarrow \E^{n,i}(X)
$$
which is functorial in $X$ and compatible with products.
\item{\emph{Gysin}. --} For any projective morphism $f:Y\to X$ between smooth $S$-schemes there is a (functorial)  morphism 
$$
f_*:\E^{n,i}(Y) \rightarrow \E^{n-2d,i-d}(X).
$$
where   $d$ is the dimension of $f$. 
\item{\emph{Projection formula}. --} 
For $f$ as above and
 any pair $(x,y) \in \E^{*,*}(X) \times \E^{*,*}(Y)$,
 one has:
$$
f_*(f^*(x).y)=x.f_*(y).
$$
\item[(3')]{\emph{Degree formula}. --} For any finite morphism $f:Y \rightarrow X$
 between smooth connected $S$-schemes,
 and any $x \in \E^{n,i}(X)$,
$$
f_*f^*(x)=d.x
$$
where $d$ is the degree of the function fields extension associated with $f$.
\item{\emph{Excess intersection formula}. --} 
Consider a cartesian square of smooth $S$-schemes:
$$
\xymatrix@=14pt{
Y'\ar^q[r]\ar_g[d] & X'\ar^f[d] \\
Y\ar^p[r] & X
}
$$
such that $p$ is projective. Let $\xi$ be the excess intersection
 bundle associated with that square  and let $e$ be its rank. 
Then for any $y \in \E^{*,*}(Y)$, one gets:
$$
f^*p_*(y)=q_*(c_e(\xi).g^*(y)).
$$
\item The regulator map $\sigma$ is natural
 with respect to the Gysin functoriality.
\item[(5')] The regulator map $\sigma$ induces a Chern character
$$
\mathrm{ch_n}:K_n(X)_\QQ \rightarrow \bigoplus_{i \in \ZZ} \E^{2i-n,i}(X)
$$
which satisfies the (higher) Riemann-Roch formula of Gillet
 (see \cite{Gillet})

\item[(6)]{\emph{Descent}. --} The cohomology $\E^{n,i}$ admits a 
 functorial extension to diagrams of $S$-schemes and satisfies
 cohomological descent for the h-topology\footnote{The h-topology
 was introduced by Voevodsky.
 Recall that covers for this topology
 are given by morphisms of schemes
 which are universal topological epimorphism.}:
 given any hypercover $p:\mathcal X \rightarrow X$ for the h-topology,
 the induced morphism:
$$
p^*:\E^{n,i}(X) \rightarrow \E^{n,i}(\mathcal X)
$$
is an isomorphism.\footnote{One deduces easily from this isomorphism
 the usual descent spectral sequence.}
\item[(7)]{\emph{Bloch-Ogus theory}. --} One can associate with $\E$
 a canonical homology theory, the Borel-Moore $\E$-homology.
 For any separated $S$-scheme $X$ with structural morphism $f$,
 and any pair of integers $(n,i)$, put:
$$
\E_{n,i}^\BM(X)=\Hom\big(\un_S,f_*f^!\E(-i)[-n]\big).
$$
Then, the pair $(\E,\E^\BM)$ is a \emph{twisted Poincar\'e duality
 theory with support} in the sense of Bloch and Ogus (cf \cite{BO}).
 Moreover Borel-Moore $\E$-homology is contravariantly functorial
 with respect to smooth morphisms.
\end{enumerate}

These properties follow easily from the results proved in \cite{CD3} and \cite{Deg8}. We collect them in Section 2.

Since our aim is to prove that rigid syntomic cohomology  satisfies the Bloch-Ogus formalism, we just need to represent it as a motivic ring spectrum. Thus we prove the following criterion, which is the main result of the first section. Before stating it we introduce the following notation: for any complex $E\in \Comp(S,\QQ)$ and $X/S$ smooth and affine let
$$	
H^n(X,E):=H^n(E(X))\  .
$$
\begin{thmi}[cf. Prop.~\ref{prop:ring_spectrum_existence}]
Let $(E_i)_{i \in \NN}$ be a family of complexes in $\Comp(S,\QQ)$ forming a $\NN$-graded commutative monoid together with a section $c:\QQ[0]\to E_1(\GG)[1]$ 
 satisfying the following properties:
\begin{enumerate}
\item{\emph{Excision}. --} Let  $E_i^{Nis}$ be the associated Nisnevich sheaves. For any integer $i$  and any $X/S$  affine  and smooth,  $H^n(X,E_i)\simeq H_{Nis}^n(X,E_i^{Nis})$. 
\item{\emph{Homotopy}. --}  For any integer $i$ and any $X/S$  affine  and smooth,  $H^n(X,E_i)\simeq H^n(\AA^1_X,E_i)$.
\item{\emph{Stability}. --} Let $\bar c$ be the image of $c$ in $H^1(\GG,E_1)$.
For any smooth $S$-scheme $X$ and any pair
 of integers $(n,i)$ the following map\footnote{We let $p:X \times \GG \rightarrow X$ be the canonical projection
    and $\pi_X$ following quotient map:
   $$
   0 \rightarrow H^n(X ,E_i)
    \xrightarrow{p^*} H^n(X \times \GG,E_i)
    \xrightarrow{\pi_X} \frac{H^{n+1}(X \times \GG,E_{i+1})}{H^{n+1}(X,E_{i+1})}
    \rightarrow 0\ .
   $$
   }
$$
H^n(X,E_i) \rightarrow \frac{H^{n+1}(X \times \GG,E_{i+1})}{H^{n+1}(X,E_{i+1})},\quad 
 x \mapsto \pi_X\big(x \times \bar c\big)\quad 
$$ 
is an isomorphism.
\item{\emph{Orientation}.--}  Let $u:\GG \rightarrow \GG$ be the inverse map of the group scheme $\GG$,
 and denote by $\bar c'$ the image of $c$ in the group
 $H^1(\GG,E_1)/H^1(S,E_1)$. The following equality holds:
$u^*(\bar c')=-\bar c'$.
\end{enumerate}
Then there exists a motivic ring spectrum $\E$  together with
 canonical isomorphisms
\begin{equation*} 
\Hom_{\DMB(S)}(M(X),\E(i)[n]) \simeq H^n(X,E_i)
\end{equation*}
for integers $(n,i) \in \ZZ \times \NN$,
functorial in the smooth $S$-scheme $X$
 and compatible with products. 
Moreover, $\E$ depends functorially
 on $(E_i)_{i\in \NN}$ and $c$.
\end{thmi}
The main difficulty of the above result is that the monoid structure on $E_i$ is defined at the level of complexes of pre-sheaves and not just in the homotopy category. Using this result we can prove (in Section 2) the existence of motivic ring  spectra representing several cohomology theories. First we prove that for any algebraic scheme $X$,
 defined over a field of characteristic zero, there is a motivic ring $\E_\FdR$ such that $\E_\FdR^{n,i}(X)\simeq F^iH^n_\dR(X)$ is the i-th step of the Hodge filtration of the de Rham cohomology of $X$ as defined by Deligne \cite{Del:TDH3}. Then we prove that the rigid cohomology of Berthelot is also represented by a motivic ring spectrum $\E_\rig$.  
As we already mentioned   the rigid syntomic cohomology of Besser is  defined using a kind of mapping cone complex whose components are differential graded algebras (namely it is the homotopy limit of the diagram in \ref{num:syntomicrecall}). Thus we cannot apply directly the above criterion since we would need to define a multiplication on  the cone compatible with that of its components. To go around this problem we prove that a homotopy limit of motivic ring spectra  is a motivic ring spectrum. Hence  the rigid syntomic cohomology can be represented by a motivic ring spectrum as claimed. 

As already mentioned, the existence of $\E_\syn$ allows us 
to naturally extend the rigid syntomic cohomology to singular schemes. By {\em devissage}, we show how to compute the syntomic cohomology of a semi-stable curve. We warn the reader that this is (probably) not the correct way to extend the cohomology to a semistable curve  in the perspective of the theory of $p$-adic $L$-functions. 

In passing we show some results about what we call the {\em absolute rigid cohomology} given by 
$$	
H^n_{\phi}(X,i):=\Hom_{D^b(\Fisoc)}(\un,R\Gamma(X)(i)[n])
$$
where $R\Gamma(X)$ is a complex of $F$-isocrystals such that $H^n(R\Gamma(X))=H_\rig^n(X)$, for $X$ a scheme over a perfect field $k$. 

\bigskip

The last application of the representability theorem of
 rigid syntomic cohomology is the existence of a natural theory
 of \emph{rigid syntomic coefficients} for $R$-schemes
 (Section \ref{sec:modules}).
Using the techniques of \cite[sec. 17]{CD3},
 we set up the theory of {\em rigid syntomic modules}:
 over any $R$-scheme $X$, they are modules (in a strict homotopical sense)
 over the ring spectrum $\E_{\syn,X}$
 obtained by pullback along the structural morphism of $X/R$.
 The corresponding category $\smod_X$ for various $R$-schemes $X$,
 shares many of the good properties of the category $\DMB$, 
 such as the complete Grothendieck six functors formalism.
 It receives a natural realization functor from $\DMB$,
  which is triangulated, monoidal (and commutes with $f^*$ and $f_!$).

This construction might be the main novelty
  of our representability theorem.
 However, to be complete we should relate these modules
 with more concrete categories of coefficients,
 probably related with $F$-isocrystals. This relation will be
 investigated in a future work.

\bigskip
\paragraph{\em Acknowledgments}
The authors are grateful to Joseph Ayoub, Amnon Besser, Bruno Chiarellotto,
Denis-Charles Cisinski, Andreas Langer, Victor Rotger and Georg Tamme, J\"org Wildeshaus for our initial motivation, their useful comments or stimulating conversations about their research in connection to this work. The second author is grateful to Michael Harris and the  ``Fondation Simone et Cino del Duca de l'Institut de France'' for the year spent in Jussieu.

\section{Motivic homotopy theory}

In this section we first recall a basic construction of motivic homotopy theory,
 the category of Morel motives (Def. \ref{df:morel}) -- the reader is referred to
 \cite{CD3} for more details.
 Then we prove a criterion for the representability of a cohomology
 theory by a ring spectrum. This criterion is new and it generalizes an
 analogous result from \cite{CD2}.

Throughout this section, $S$ will be a base scheme,
 assumed to be noetherian finite dimensional and $\Lambda$
 will be a ring of coefficients.
We will denote by $\smx S$ either the category of smooth
 $S$-schemes of finite type or
 the category of such schemes which in addition are affine
 (absolutely). Note that equipped with the Nisnevich topology,
 the two induced topoi are equivalent. 

\subsection{The effective $\AA^1$-derived category}

\begin{num}
We let $\PSh(S,\Lambda)$ be the category of presheaves of $\Lambda$-modules
 on $\smx S$ and $\Comp(\PSh(S,\Lambda))$ the category of complexes
 of such presheaves. Given such a complex $K$, a smooth $S$-scheme $X$
 and an integer $n \in \ZZ$, we put:
$$
H^n(X,K):=H^n(K(X)).
$$
This is the cohomology of $K$ computed in the derived category of
 $\PSh(S,\Lambda)$: 
 if we denote by $\Lambda(X)$ the presheaf of $\Lambda$-modules represented by $X$,
 we get:
$$
H^n(X,K)=\Hom_{\Der(\PSh(S,\Lambda))}(\Lambda(X),K[n]).
$$

A closed pair will be a couple $(X,Z)$ such that $X$ is
 a smooth $S$-scheme
 and $Z$ is a closed subscheme of $X$
  -- in fact one requires that $X$ and $(X-Z)$ are in $\smx S$.
 We also define the $n$-th cohomology group of $(X,Z)$ 
 -- equivalently: of $X$ with support in $Z$ --
 with coefficients in $K$ as:
$$
H^n_Z(X,K):=H^{n-1}\big(\mathrm{Cone}(K(X) \rightarrow K(X-Z))\big).
$$
A morphism of closed pairs $f:(Y,T) \rightarrow (X,Z)$
 is a morphism of schemes $f:Y \rightarrow X$ such that
 $f^{-1}(Z) \subset T$. We say $f$ is \emph{excisive}  if it is étale, $f^{-1}(Z)=T$ and $f$ induces an isomorphism $T_{red}\to Z_{red}$.
 The cohomology groups $H^*_Z(X,K)$ are contravariant
 in $(X,Z)$ with respect to morphisms of closed pairs.
\end{num}
\begin{df}
Let $K$ be a complex of $\PSh(S,\Lambda)$.
\begin{enumerate}
\item We say that $K$ is \emph{Nis-local}
 if for any excisive morphism of closed pairs
  $f:(Y,T) \rightarrow (X,Z)$, the pullback morphism
$$
f^*:H_Z^*(X,K) \rightarrow H_T^*(Y,K)
$$
is an isomorphism.
\item We say that $K$ is \emph{$\AA^1$-local}
 if for any smooth $S$-scheme $X$,
 the pullback induced by the canonical projection $p$
 of the affine line over $X$
$$
p^*:H^*(X,K) \rightarrow H^*(\AA^1_X,K)
$$
is an isomorphism.
\end{enumerate}
Following Morel,
 we define the \emph{effective $\AA^1$-derived category
 over $S$ with coefficients in $\Lambda$}
 as the full subcategory of $\Der(\PSh(S,\Lambda))$ made
 by complexes which are Nis-local and $\AA^1$-local.
We will denote it by $\DMte(S,\Lambda)$.
\end{df}

\begin{num}
Let us recall the following facts on the category defined above:
\begin{enumerate}
\item Let $\Sh(S,\Lambda)$ be the category of sheaves of $\Lambda$-modules on
 $\smx S$ for the Nisnevich topology. 
 Then $\DMte(S,\Lambda)$ is equivalent to the $\AA^1$-localization
  of the derived category $\Der(\Sh(S,\Lambda))$,
  as defined in \cite[\S~1.1]{CD2}.

This comes from the fact that the pair of adjoint functors,
 whose left adjoint is the associated Nisnevich sheaf $a$,
 induces a derived adjunction
$$
a:\Der(\PSh(S,\Lambda)) \leftrightarrows \Der(\Sh(S,\Lambda)):\mathcal O
$$
whose right adjoint $\mathcal O$ is fully faithful
 with essential image the complexes which are Nis-local
 -- this is classical see for example \cite[5.2.10 and 5.2.13]{CD3}.
 In particular, Nis-local complexes can be described as
 those complexes $K$ which satisfy Nisnevich descent:
 for any Nisnevich hypercover $P_\bullet \rightarrow X$
 of any smooth $S$-scheme $X$, the induced map:
$$
K(X) \rightarrow \Tot\big(K(P_\bullet)\big)
$$
is a quasi-isomorphism -- the right hand-side
 is the total complex associated
 with the obvious double complex.
\item The fact that the category $\DMte(S,\Lambda)$ 
 can be handled in practice comes
 from its description as the homotopy category
 associated with an explicit model category structure on
 the category $\Comp(\PSh(S,\Lambda))$ of complexes on 
 the Grothendieck abelian category $\PSh(S,\Lambda)$:
\begin{itemize}
\item \emph{Weak equivalences}
 (also called weak $\AA^1$-equivalences)
 are the morphisms of complexes $f$
 such that for any complex $K$ which is $\AA^1$-local
 and Nis-local, $\Hom_{\Der(\PSh(S,\Lambda))}(f,K)$ is
  an isomorphism.
\item Fibrant objects are the complexes which are Nis-local
 and $\AA^1$-local. \emph{Fibrations} are the morphisms of complexes
 which are epimorphisms and whose kernel is fibrant. 
\end{itemize}
For the proof that this defines a model category,
 we refer the reader to \cite{CD1}: we first consider
 the model category structure associated with the
 Grothendieck abelian category $\PSh(S,\Lambda)$
 (see \cite[Ex. 2.3]{CD1})
 and we localize it with respect to Nisnevich hypercovers
 and $\AA^1$-homotopy (\cite[Section 4]{CD1}).
 Let us recall that a typical example of \emph{cofibrant objects}
 for this model structure are the presheaves of the form
 $\Lambda(X)$ for a smooth $S$-scheme $X$.

We derive from this model structure the existence
 of fibrant (resp. cofibrant) resolutions:
 associated with a complex of presheaves $K$,
 we get a fibrant $K_f$ (resp. cofibrant $K_c$) and a map
$$
K \rightarrow K_f\qquad \text{ (resp. } K_c \rightarrow K),
$$
which is a cofibration (resp. fibration)
 and a weak $\AA^1$-equivalence. 
 These resolutions can be chosen to be natural in $K$.

This can be used to derive functors.
 In particular, the natural tensor product $\otimes$
 of $\Comp(\PSh(S,\Lambda))$ as well as its internal complex morphism
 $\uHom$ can be derived using the formulas:
$$
K \otimes^\derL L=K_c \otimes L_c,\qquad
 \derR \uHom(K,L)=\uHom(K_c,L_f);
$$
see \cite[Sections 3 and
 4]{CD1}.\footnote{\label{fn:monoid_axiom} Note in particular that, 
 according to \cite[Proposition~4.11]{CD1},
 the model category described above is a monoidal
 model category which satisfies the monoid axiom.}
\end{enumerate}
\end{num}

\subsection{The $\AA^1$-derived category}

\begin{num} \label{num:graded_dga}
We define the \emph{Tate object} as the following complex
 of presheaves of $\Lambda$-modules:
\begin{equation} \label{eq:Tate}
\Lambda(1):=\mathrm{coKer}(\Lambda \xrightarrow{s_{1*}} \Lambda(\GG))[-1]
\end{equation}
where $s_1$ is the unit section of the group scheme $\GG$,
 considered as an $S$-scheme.
 Given a complex $K$ and an integer $i \geq 0$,
 we denote by $K(i)$ the tensor product of $K$ with
 the $i$-th tensor power of $\Lambda(1)$ (on the right).

As usual in the general theory of motives,
 one is led to invert the object $\Lambda(1)$ for the tensor product.
 In the context of motivic homotopy theory,
  this is done using the construction of spectra,
  borrowed from algebraic topology.

For any integer $i>0$,
 we will denote by $\Sigma_i$ the group of permutations
 of the set $\{1,\cdots,i\}$, $\Sigma_0=1$.
\end{num}
\begin{df} \label{df:tate_spectra}
A \emph{Tate spectrum} (over $S$ with coefficients in $\Lambda$),
 is a sequence $\E=(E_i,\sigma_i)_{i \in \NN}$
 such that:
\begin{itemize}
\item for each $i \in \NN$, $E_i$ is a complex of $\PSh(S,\Lambda)$
 equipped with an action of $\Sigma_i$,
\item for each $i \in \NN$, $\sigma_i$ is a morphism of complexes
$$
\sigma_i:E_i(1) \rightarrow E_{i+1},
$$
called the \emph{suspension map} (in degree $n$).
\item For any integers $i \geq 0$, $r>0$,
 the map induced by the morphisms $\sigma_i,\cdots,\sigma_{i+r}$:
$$
E_i(r) \rightarrow E_{i+r}
$$
is compatible with the action of $\Sigma_i \times \Sigma_r$,
 given on the left by the structural $\Sigma_i$-action
  on $E_i$ and the action of $\Sigma_r$ via the permutation
  isomorphism of the tensor structure on $\Comp(\PSh(S,\Lambda))$,
 and on the right via the embedding
  $\Sigma_i \times \Sigma_r \rightarrow \Sigma_{i+r}$
	 obtained by identifying the sets $\{1,...,i+r\}$ and
	 $\{1,...,i\} \sqcup \{1,...,r\}$.
\end{itemize}
A morphism of Tate spectra $f:\E \rightarrow \F$ is a sequence
 of  $\Sigma_i$-equivariant maps
  $(f_i:E_i \rightarrow F_i)_{n \in \NN}$ 
 compatible with the suspension maps.
 The corresponding category will be denoted by $\Sp(S,\Lambda)$.

A morphism $f$ as above is called a \emph{level weak equivalence}
 if for any integer $i \geq 0$, the morphism of complexes $f_i$
 is a quasi-isomorphism. We denote by $D_{\rm Tate}(S,\Lambda)$
 the localization of $\Sp(S,\Lambda)$ with respect to
 level weak equivalences (See \cite[Sec. 1.4]{CD2}).
\end{df}
Complexes and spectra are linked by a pair of adjoint functors
 $(\Sigma^\infty,\Omega^\infty)$ defined respectively for 
 a complex $K$ and a Tate spectrum $\E$ as follows:
\begin{equation} \label{eq:suspension}
\Sigma^\infty K:=(K(i))_{i \in \NN}\ , \qquad \Omega^\infty(\E)=E_0,
\end{equation}
where $K(i)$ is equipped with the action of $\Sigma_i$
 by its natural action  through the symmetry isomorphism
 of the tensor structure on $\Comp(\PSh(S,\Lambda))$.

\begin{num} \label{num:product_spectra}
The category of Tate spectra can be described using the category
 of symmetric sequences of $\Comp(\PSh(S,\Lambda))$: the objects
 of this category
 are the sequences $(E_i)_{i \in \NN}$ of complexes of $\PSh(S,\Lambda)$
 such that $E_i$ is equipped with an action of $\Sigma_i$.
 This is a Grothendieck abelian category on which on can construct
 a closed symmetric monoidal structure (see \cite[Section 7]{CD1}).
 Moreover, the obvious symmetric sequence 
$$
\mathrm{Sym}(\Lambda(1)):=(\Lambda(i))_{i \in \NN}
$$
 has a canonical structure of a commutative monoid.

The category $\Sp(S,\Lambda)$ is equivalent to the category of modules
 over $\mathrm{Sym}(\Lambda(1))$ (see again \emph{loc. cit.}).
 Therefore, it is formally a Grothendieck abelian category
 equipped with a closed symmetric monoidal structure.
 Note that the tensor product can be described by the following
 universal property:
 to give a morphism of Tate spectra $\mu:\E \otimes \F \rightarrow \G$
 is equivalent to give a family of morphisms
$$
\mu_{i,j}:E_i \otimes F_j \rightarrow G_{i+j}
$$
which is $\Sigma_i \times \Sigma_j$-equivariant
 and compatible with the suspension maps
 (see \emph{loc. cit.} Remark 7.2).
\end{num}

\begin{df}
Let $\E$ be a Tate spectrum over $S$ with coefficients in $\Lambda$.
\begin{enumerate}
\item We say that $\E$ is \emph{Nis-local}
 (resp. \emph{$\AA^1$-local})
 if for any integer $i \geq 0$, the complex $E_i$
 is Nis-local (resp. $\AA^1$-local).
\item We say that $\E$ is a \emph{Tate $\Omega$-spectrum}
 if the morphism of $\DMte(S,\Lambda)$ induced by adjunction
 from $\sigma_i$:
$$
E_i \rightarrow \derR \uHom(\Lambda(1),E_{i+1})
$$
is an isomorphism (i.e. a weak $\AA^1$-equivalence).
\end{enumerate}
For short, we say that $\E$ is \emph{stably fibrant}
 if it is an $\Omega$-spectrum which is Nis-local
 and $\AA^1$-local.

We define the \emph{$\AA^1$-derived category
 over $S$ with coefficients in $\Lambda$},
  denoted by $\DMt(S,\Lambda)$,
 as the full subcategory of $D_{\rm Tate}(S,\Lambda)$ made
 by the stably fibrant Tate spectra.
\end{df}

\begin{num} \label{num:recall_DM}
Recall the following facts
 on the previous construction:
\begin{enumerate}
\item The construction of $\DMt(S,\Lambda)$ through spectra
 is a classical construction derived from algebraic
 topology (see \cite{Hov}).
 In particular, the monoidal model structure on the category
 $\Comp(\PSh(S,\Lambda))$
 induces a canonical monoidal model structure on $\Sp(S,\Lambda)$ whose
 homotopy category is precisely $\DMt(S,\Lambda)$.
 It is called the \emph{stable model category}.

Therefore $\DMt(S,\Lambda)$ is a symmetric monoidal triangulated category
 with internal Hom. 
 Moreover, the adjoint functors \eqref{eq:suspension} 
 can be derived:
\begin{equation}\label{eq:sus_omega}
\sus:\DMte(S,\Lambda) \leftrightarrows \DMt(S,\Lambda):\Omega^\infty.
\end{equation}
The functor $\Sigma^\infty$ is monoidal.\footnote{In fact, 
 the homotopy category $\DMt(S,\Lambda)$ equipped with its left derived
 functor $\Sigma^\infty$, is universal for the property that
 $\Sigma^\infty$ is monoidal and $\Sigma^\infty(K(1))$ is
 $\otimes$-invertible (see again \cite{Hov}).}
Recall also that given a Tate $\Omega$-spectrum $\E$ as above
 and an integer $i\geq 0$, we get:
\begin{equation}\label{eq:omega}
\Omega^\infty(\E(i))=E_i.
\end{equation}
We will simply denote by $\Lambda$ or $\un$ the unit of $\DMt(S,\Lambda)$
 -- instead of $\sus \Lambda$.
\item In fact the triangulated categories of the form $\DMt(S,\Lambda)$ 
 for various schemes $S$
 are not
 only closed monoidal but they are equipped with the complete
 formalism of Grothendieck six operations 
$$
(f^*,f_*,f_!,f^!,\otimes,\uHom)
$$
as established by Ayoub in \cite{Ayoub}.\footnote{Ayoub
 treats only the case where $f$ is quasi-projective
 for the existence of the adjoint pair $(f_!,f^!)$.
 The general case can be obtained by using 
 the classical construction of Deligne as explained in
 \cite[section 2.2]{CD3}. The reader will also find
 a summary of the six operations formalism in \emph{loc. cit.}
 Theorem 2.4.50.} 
\end{enumerate}
\end{num}

\subsection{Triangulated mixed motives}

\begin{num} \label{num:Morel}
\underline{In this section, $\Lambda$ is a $\QQ$-algebra.}

We recall the construction of Morel
 for deriving the triangulated category of mixed motives
 from the category $\DMt(S,\Lambda)$ (see \cite[16.2]{CD3}
 for details).

Let us consider the inverse map $u$ of the multiplicative group
 scheme $\GG$, corresponding to the map:
$$
\mathcal O_S[t,t^{-1}] \rightarrow \mathcal O_S[t,t^{-1}]\ ,\qquad
 t \mapsto t^{-1}.
$$
Recall from formula \eqref{eq:Tate} the decomposition
 $\Lambda(\GG)=\Lambda \oplus \Lambda(1)[1]$, considered in $\DMt(S,\Lambda)$.
Given this decomposition,
 the map $u_*:\Lambda(\GG) \rightarrow \Lambda(\GG)$ can be written in matrix
 form as:
$$
\begin{pmatrix} 1 & 0 \\ 0 & \epsilon_1\end{pmatrix}
$$
Because $\Lambda(1)[1]$ is $\otimes$-invertible in $\DMt(S,\Lambda)$,
 there exists a unique endomorphism $\epsilon$ of $\Lambda$ in $\DMt(S,\Lambda)$
 such that $\epsilon_1=\epsilon(1)[1]$.

Because $u^2=1$, we get $\epsilon^2=1$. Thus we can define two
 complementary projectors in $\End_{\DMt(S,\Lambda)}(\Lambda)$:
$$
p_+=\frac 1 2.(1_\Lambda-\epsilon), \qquad p_-=\frac 1 2.(1_\Lambda+\epsilon).
$$
Given any object $\E$ in $\DMt(S,\Lambda)$, we deduce projectors
 $p_+ \otimes \E$, $p_- \otimes \E$ of $\E$. Because $\DMt(S,\Lambda)$
 is pseudo-abelian\footnote{This is, for example,
 an application of the fact it is a triangulated
 category with countable direct sums (cf \cite[1.6.8]{Nee}).}, 
 we deduce a canonical decomposition:
$$
\E=\E_+ \oplus \E_-
$$
where $\E_+$ (resp. $\E_-$) is the image of $p_+ \otimes E$
 (resp. $p_- \otimes E$).
 The following triangulated category was introduced by Morel.
\end{num}
\begin{df} \label{df:morel}
An object $\E$ in $\DMt(S,\Lambda)$ will be called a Morel
 motive if $\E_-=0$.
We denote by $\DMt(S,\Lambda)_+$ the full subcategory of $\DMt(S,\Lambda)$
 made by Morel motives.
\end{df}
Note that according to the above, the fact $\E$ is a Morel
 motive is equivalent to the property:
\begin{equation} \label{eq:Morel_mot}
\epsilon \otimes \E=-1_\E;
\end{equation}
in other words, $\epsilon$ acts as $-1$ on $\E$.

\begin{num}\label{num:DMB}
Recall the following facts, which legitimate
 the terminology of ``Morel motives":
\begin{enumerate}
\item Obviously,
 the category $\DMt(S,\Lambda)_+$ is a triangulated monoidal sub-category
 of $\DMt(S,\Lambda)$. Moreover, the six operations on $\DMt(-,\Lambda)$
 induce similar operations on $\DMt(-,\Lambda)_+$
 which satisfy all of the six functors formalism.
\item According to \cite[16.2.13]{CD3}, there is an equivalence
 of triangulated monoidal categories:
$$
\DMt(S,\Lambda)_+ \simeq \DMB(S,\Lambda)
$$
where $\DMB(S,\Lambda)$ is the triangulated category of Beilinson
 motives introduced in \cite[Def. 14.2.1]{CD3}.
In $\DMB(S,\Lambda)$, given a smooth $S$-scheme $X$,
 we simply denote by $M(X)$ the object corresponding to $\sus \Lambda(X)$
 and call it the \emph{motive} of $X$. \\
 Concretely, the above isomorphism means that when $S$ is regular,
 for any smooth $S$-scheme $X$ and any pair $(n,i) \in \ZZ^2$,
 one has a canonical isomorphism:
\begin{equation}
\Hom_{\DMt(S,\Lambda)_+}(\sus \Lambda(X),\Lambda(i)[n])
 \simeq K_{2i-n}^{(i)}(X) \otimes_\QQ \Lambda
\end{equation}
where $K_{2i-n}^{(i)}(X)$ denotes the $i$-th Adams subspace
 of the rational Quillen K-theory of $X$ in homological degree 
 $(2i-n)$.\footnote{This formula was first obtained by Morel but 
 the proof has not
 been published. In any case, this is a consequence of 
 \emph{loc. cit.}}

Note in particular that according to the coniveau spectral
 sequence in K-theory and a computation of Quillen,
 a particular case of the above isomorphism is the 
 following one:
\begin{equation}
\Hom_{\DMt(S,\Lambda)_+}(\sus \Lambda(X),\Lambda(n)[2n])
 \simeq CH^n(X) \otimes_\ZZ \Lambda\ ,
\end{equation}
where the right hand side is the Chow group of $n$-codimensional
 $\Lambda$-cycles in $X$
  ($S$ is still assumed to be regular).
\end{enumerate}
\end{num}
\subsection{Ring spectra}
\begin{num} \label{num:monoid}
Recall that a commutative monoid in a symmetric monoidal category
 $(\mathcal M,\otimes,\un)$ is an object $M$,
 a unit map $\eta:\un \rightarrow M$ and a multiplication
map $\mu:M \otimes M \rightarrow M$, such that the following diagrams
are commutative:
$$
\xymatrix@R=10pt{
\ar@{}|{\text{Unit:}}[r] &&\ar@{}|{\text{Associativity:}}[r]
&&\ar@{}|{\text{Commutativity:}}[r] & \\
M\ar^-{1 \otimes \eta}[r]\ar@{=}[rdd] & M \otimes M\ar^\mu[dd]
 & M \otimes M \otimes M\ar^-{1 \otimes \mu}[r]\ar_{\mu \otimes 1}[dd]
  &  M \otimes M\ar^\mu[dd]
 &  M \otimes M\ar^\mu[rd]\ar_\gamma[dd] & \\
& & & & & M \\
& M
 & M \otimes M\ar^-\mu[r] & M
 & M \otimes M\ar_\mu[ru]
}
$$
where $\gamma$ is the obvious symmetry isomorphism.
\end{num}

\begin{df}
A \emph{weak ring spectrum} (resp. \emph{ring spectrum}) $\E$ over $S$
 is a commutative monoid 
 in the symmetric monoidal category $\DMt(S,\Lambda)$
 (resp. $\Sp(S,\Lambda)$).\footnote{Ring spectra have slowly emerged
  in homotopy theory and the terminology is not fixed. Usually,
  our weak ring spectra (resp. ring spectra)
  are simply called ring spectra (resp. highly structured
  ring spectra).}
\end{df}

\begin{num}
A spectrum $\E$ in $\DMt(S,\Lambda)$ defines a bigraded cohomology theory 
 on smooth $S$-schemes $X$ by the formula:
$$
\E^{n,i}(X)=\Hom_{\DMt(S,\Lambda)}(\sus \Lambda(X),\E(i)[n]).
$$
The structure of a weak ring spectrum on $\E$ corresponds to a product 
 in cohomology, usually called the cup-product and defined as follows: 
 given cohomology classes,
$$
\alpha:\sus \Lambda(X) \rightarrow \E(i)[n],
 \qquad \beta:\sus \Lambda(X) \rightarrow \E(j)[m]
$$
one defines the class $\alpha \cupp \beta$ as the following
composite:
$$
\sus \Lambda(X) \xrightarrow{\delta_*}
 \sus \Lambda(X) \otimes \sus \Lambda(X)
 \xrightarrow{\alpha \otimes \beta} \E(i)[n] \otimes \E(j)[m]
 \xrightarrow{\mu} \E(i+j)[n+m]. 
$$
Using this definition, one can check easily 
 that the commutativity axiom of $\E$ implies the following
 formula:
\begin{equation*} \label{eq:product}
\alpha \cupp \beta=(-1)^{nm-ij}.\epsilon^{ij}.\beta \cupp \alpha
\end{equation*}
where $\epsilon$ is the endomorphism of $\Lambda$ introduced in Paragraph
 \ref{num:Morel}. In particular, if $\E$ is a Morel motive,
 the product on $\E^{**}$ is anti-commutative
 with respect to the first index and commutative with respect
 to the second one. Note also the following result which will
 be used later.
\end{num}
\begin{lm} \label{lm:ring_Morel}
Let $\E$ be a weak ring spectrum  with unit $\eta$
 and multiplication $\mu$.
 Then the following conditions are equivalent:
\begin{enumerate}
\item[(i)] $\E$ is a Morel motive.
\item[(ii)] $\eta \circ \epsilon=-\eta$.
\end{enumerate}
\end{lm}
\begin{proof}
Let us remark that according to the Unit property
 the following equalities hold:
\begin{align*}
\mu \circ (1_\E \otimes \eta)&=1_\E, \\
\mu \circ (1_\E \otimes (\eta \circ \epsilon))&=\epsilon \otimes \E.
\end{align*}
Thus the equivalence between (i) and (ii)
 directly follows from relation \eqref{eq:Morel_mot}
 characterizing Morel motives.
\end{proof}

\begin{rem}
Of course, a ring spectrum induces a weak ring spectrum.
 Concretely, in the non weak case, one requires
 that the diagrams of Paragraph \ref{num:monoid} commutes
 in the mere category of spectra, and not only up to weak homotopy.
 This makes the construction of ring spectra
 more difficult than usual weak ring spectra.
\end{rem}

\begin{num} \label{num:ring_spectra}
Let us denote by $\Spr(S,\Lambda)$ the category of ring spectra.
Because the category $\Sp(S,\Lambda)$ is a complete and cocomplete
monoidal category, $\Spr(S,\Lambda)$ is complete and cocomplete.
Moreover, the forgetful functor:
$$
U:\Spr(S,\Lambda) \rightarrow \Sp(S,\Lambda)
$$
admits a left adjoint which we denote by $F$.
The following result appears in \cite[Th. 7.1.8]{CD3}.
\end{num}
\begin{thm} \label{thm:model_ring_sp}
Assume $\Lambda$ is a $\QQ$-algebra.

Then the category $\Spr(S,\Lambda)$ is a model category
 whose weak equivalences (resp. fibrations)
 are the maps $f$ such that $U(f)$ is a weak equivalence
 (resp. stable fibration) in the stable model category
 $\Sp(S,\Lambda)$ (see Par. \ref{num:recall_DM}).

We denote by $\Ho(\Spr(S,\Lambda))$ the homotopy category
 associated with this model category.
\end{thm}

\begin{num}\label{num:homotopy_limits} For a given $\QQ$-algebra $\Lambda$,
 recall the following consequences of this theorem:
\begin{enumerate}
\item The pair of adjoint functors $(F,U)$ can be derived
 and induces adjoint functors:
$$
\derL F:\DMt(S,\Lambda) \leftrightarrows \Ho(\Spr(S,\Lambda)):U
$$
The essential image of the functor $U$ lies in the category
 of weak ring spectra. However, it is not essentially
 surjective on that category.
\item As any homotopy category of a model category, 
 the homotopy category $\Ho(\Spr(S,\Lambda))$
 admits homotopy limits and colimits
 (see \cite[Intro. Th. 1]{Cis}). In other words,
 any diagram of $\Spr(S,\Lambda)$ admits a homotopy limit
 and a homotopy colimit.
\end{enumerate}
\end{num}

\begin{num} \label{num:axioms_comm_monoid}
A commutative monoid in the category
 $\Comp(\PSh(S,\Lambda))$ is usually called a commutative differential
 graded $\Lambda$-algebra with coefficients in the abelian monoidal category
 $\PSh(S,\Lambda)$.

A $\NN$-graded commutative monoid in $\Comp(\PSh(S,\Lambda))$
 is a sequence $(E_i)_{i \in \NN}$ of complexes of pre\-shea\-ves
 equipped with a unit map $\eta:\Lambda \rightarrow E_0$
 and multiplication maps $\mu_{ij}:E_i \otimes E_j \rightarrow E_{i+j}$
 for any pair of integers $(i,j)$ such that the following diagrams
 commute:
$$
\xymatrix@R=10pt{
\ar@{}|{\text{Unit:}}[r] &&\ar@{}|{\text{Associativity:}}[r]
&&\ar@{}|{\text{Commutativity:}}[r] & \\
E_i\ar^-{1 \otimes \eta}[r]\ar@{=}[rdd] & E_i \otimes E_0\ar^{\mu_{i,0}}[dd]
 & E_i \otimes E_j \otimes E_k\ar^-{1 \otimes \mu_{jk}}[r]
    \ar_{\mu_{ij} \otimes 1}[dd]
  &  E_i \otimes E_{j+k}\ar^{\mu_{i,j+k}}[dd]
 &  E_i \otimes E_j\ar^{\mu_{ij}}[rd]\ar_{\gamma_{ij}}[dd] & \\
& & & & & E_{i+j} \\
& E_i
 & E_{i+j} \otimes E_k\ar^-{\mu_{i+j,k}}[r] & E_{i+j+k}
 & E_j \otimes E_i\ar_{\mu_{j,i}}[ru]
}
$$
where $\gamma_{ij}$ is the obvious symmetry isomorphism.
We then define bigraded cohomology groups
 for any smooth $S$-scheme $X$ and any couple of integers $(n,i)$:
$$
H^n(X,E_i)=H^n(E_i(X)).
$$
The above monoid structure induces an exterior product
 on these cohomology groups:
$$
H^n(X,E_i) \otimes H^m(Y,E_j)
 \rightarrow H^{n+m}(X \times_S Y,E_{i+j}),\quad 
  (x,y) \mapsto x \times y.
$$
Given any smooth $S$-scheme $X$,
 we let $p:X \times \GG \rightarrow X$ be the canonical projection
 and consider for the next statement the following split exact sequence:
$$
0 \rightarrow H^n(X ,E_i)
 \xrightarrow{p^*} H^n(X \times \GG,E_i)
 \xrightarrow{\pi_X} \tilde H^n(X \times \GG,E_i)
 \rightarrow 0\ ,
$$
where $\tilde H^n(X \times \GG,E_i):=\Coker(p^*)$ and $\pi_X$ is the canonical projection.
\end{num}
\begin{prop}\label{prop:ring_spectrum_existence}
Suppose given a $\NN$-graded commutative monoid
 $(E_i)_{i \in \NN}$ in $\Comp(\PSh(S,\Lambda))$ 
 as above together with a section $c$ of $E_1[1]$ over $\GG$
 satisfying the following properties:
\begin{enumerate}
\item{\emph{Excision}.--} For any integer $i$, $E_i$ is Nis-local.
\item{\emph{Homotopy}.--}  For any integer $i$, $E_i$ is $\AA^1$-local.
\item{\emph{Stability}.--} Let $\bar c$ be the image of $c$ in $H^1(\GG,E_1)$.
For any smooth $S$-scheme $X$ and any pair
 of integers $(n,i)$ the following map
$$
H^n(X,E_i) \rightarrow \tilde H^{n+1}(X \times \GG,E_{i+1}),\quad
 x \mapsto \pi_X\big(x \times \bar c\big)
$$
is an isomorphism.
\end{enumerate}
Then there exists a ring spectrum $\E$ which is 
 a stably fibrant Tate spectrum together with
 canonical isomorphisms
\begin{equation} \label{eq:compute_rep}
\Hom_{\DMt(S,\Lambda)}(\sus \Lambda(X),\E(i)[n]) \simeq H^n(X,E_i)
\end{equation}
for integers $(n,i) \in \ZZ \times \NN$,
functorial in the smooth $S$-scheme $X$
 and compatible with products. 
Moreover, $\E$ depends functorially
 on $(E_i)_{i\in \NN}$ and $c$.

Assume $\Lambda$ is a $\QQ$-algebra.
Let $u:\GG \rightarrow \GG$ be the inverse map of the group scheme $\GG$,
 and denote by $\bar c'$ the image of $c$ in the group
 $\tilde H^1(\GG,E_1)$.
Then, under the above assumptions,
 the following conditions are equivalent:
\begin{enumerate}
\item[(i)] The Tate spectrum $\E$ is a Morel motive
 (\emph{i.e.} defines an object in $\DMB(S,\Lambda)$,
 Def. \ref{df:morel} and Par. \ref{num:DMB}).
\item[(ii)] The following equality holds in $\tilde H^1(\GG,E_1)$:
$u^*(\bar c')=-\bar c'$.
\end{enumerate}
\end{prop}

\begin{rem} \label{rem:axioms_motivic_ring_sp}
\begin{enumerate}
\item The two last properties should be called
 the \emph{Orientation} property. In fact, they can be reformulated
 by saying that $\E$ is an oriented ring spectrum
  (cf \cite[Cor. 14.2.16]{CD3}). Recall also this is equivalent
  to the existence of a canonical morphism of groups:
$$
\Pic(X) \rightarrow H^2(X,E_1)
$$
which is functorial in $X$ (and even uniquely determined by $c$).
\item The Stability axiom can be reformulated by
 saying that for any $x \in H^{n+1}(X \times \GG,E_{i+1})$
 there exists a unique couple 
 $(x_0,x_1) \in H^{n+1}(X,E_{i+1}) \times H^n(X,E_i)$
 such that:
$$
x=p^*(x_0)+x_1 \times \bar c'.
$$
\item Though we start with a positively graded complex $(E_i)_{i \in \NN}$
 we get a cohomology theory which possibly has negative twists.
 These negative twists are given by the following short exact sequence
 for $i>0$:
$$
0 \rightarrow \E^{n,-i}(X)
 \rightarrow H^n(X \times \GG^{i},E_0)
 \rightarrow H^n(X \times \GG^{i-1},E_0)
 \rightarrow 0
$$
 where the epimorphism is given by the sum of the inclusions 
$$
\GG^{i-1} \rightarrow \GG^{i}
$$
corresponding to set one of the coordinates of the target to $1$.
\end{enumerate}
\end{rem}

\begin{proof}
We define the Tate spectrum $\E$ to be the complex of presheaves
 $E_i$ in degree $i$ with trivial action of $\Sigma_i$.
 The section $c$ defines a map of presheaves:
$$
c':\Lambda(1) \rightarrow \Lambda(\GG)[-1] \xrightarrow{c[-1]} E_1
$$
where the first map is given by the canonical inclusion.
We define the suspension map of $\E$ in degree $i$ as the following
composite:
$$
\sigma_i:E_i(1)=E_i \otimes \Lambda(1)
 \xrightarrow{1 \otimes c'} E_i \otimes E_1
 \xrightarrow{\mu_{i,1}} E_{i+1}.
$$
One deduces from the commutative diagram
 called ``Commutativity'' of Paragraph \ref{num:axioms_comm_monoid}
  that the induced map
 $E_i(r) \rightarrow E_{i+r}$
 is $\Sigma_i \times \Sigma_r$-equivariant.
 So that $\E$ is indeed a Tate spectrum.

By definition, Assumptions (1) and (2) exactly say that $\E$
 is Nis-local and $\AA^1$-local. It remains to check it
 is an $\Omega$-spectra. In other words, the map obtained by adjunction
 from $\sigma_i$
$$
\sigma'_i:E_i \rightarrow \derR \uHom(\Lambda(1),E_{i+1})
$$
is an isomorphism in $\DMte(S,\Lambda)$. It is sufficient to check
 that for any smooth $S$-scheme $X$ and any integer $n \in \ZZ$,
 the induced map
 :
\begin{align*}
\sigma'_{i*}:
\Hom(\Lambda(X),E_i[n])
 \rightarrow 
&\Hom(\Lambda(X),\derR \uHom(\Lambda(1),E_{i+1}[n]) \\
&=\Hom(\Lambda(X) \otimes \Lambda(1),E_{i+1}[n]),
\end{align*}
 where the morphisms are taken in $\DMte(S,\Lambda)$,
 is an isomorphism.
 According to the definition, we can compute this map as follows:
\begin{equation}\label{eq:stab_iso}
\Hom\big(\Lambda(X),E_i[n]\big)
 \rightarrow \Hom\big(\Lambda(X) \otimes \Lambda(1),E_{i+1}[n]\big),
  x \mapsto x \times \bar c'
\end{equation}
where $\bar c'$ is the class of the map $c'$ in $\DMte(S,\Lambda)$.
Using the fact $E_i$ is Nis-local and $\AA^1$-local,
 the source of this map is isomorphic to $H^n(X,E_i)$.
 Similarly, the group of morphisms
$$
\Hom\big(\Lambda(X) \otimes \Lambda(\GG),E_{i+1}[n+1]\big)
$$
 is isomorphic to $H^{n+1}(X \times \GG,E_{i+1})$.
 Under this isomorphism, the target of the above map
  corresponds to $\tilde H^{n+1}(X \times \GG,E_{i+1})$.
 Under these identifications, $\bar c'=\pi_X(\bar c)$.
 Thus, the fact $\sigma'_i$ is an isomorphism
 directly follows from Assumption (3).

According to this construction, 
the maps $\eta$ and $\mu_{ij}$ induces a structure
of a ring spectrum on $\E$ (using in particular the description
 of the tensor product of spectra recalled
 in Paragraph \ref{num:product_spectra}).

The isomorphism \eqref{eq:compute_rep} follows using
 the adjunction \eqref{eq:sus_omega} and the relation
 \eqref{eq:omega} applied to the Tate $\Omega$-spectrum $\E$.
 The fact it is functorial and compatible with products
 is obvious from the above construction.

\bigskip

Let us finally consider the remaining assertion.
Note that according to what was just said,
 the class $\bar c'$ introduced in the beginning of the proof
 coincides with the class $\bar c'$ which appears
 in the statement of the proposition.
Under the isomorphism \eqref{eq:compute_rep},
 the canonical isomorphism:
$$
\Hom_{\DMt(S,\Lambda)}(\Lambda,\E) \rightarrow \Hom_{\DMt(S,\Lambda)}(\Lambda(1),\E(1))
$$
corresponds to an isomorphism of the form
$$
\Hom_{\DMte(S,\Lambda)}(\Lambda,E_0)
 \rightarrow \Hom_{\DMte(S,\Lambda)}(\Lambda(1),E_{1})=\tilde H^1(\GG,E_1)
$$
which is a particular case of the isomorphism \eqref{eq:stab_iso}
 considered above. Thus, it sends the unit map $\eta$ of $\E$
 to the class $\bar c'$. Thus the equivalence of conditions
 (i) and (ii) follows from Lemma \ref{lm:ring_Morel}.
\end{proof}

\begin{rem}
This proposition is an extension of the construction given
 in \cite[sec. 2.1]{CD2}. The main difference is that we consider
 here theories in which the different twists are not necessarily 
 isomorphic. By contrast, we require the datum of a stability class 
 here whereas we do not need a particular choice in \emph{op. cit.}

Note also that a similar extension has appeared in \cite{HS}
 applied to Deligne cohomology.
\end{rem}
\section{Motivic ring spectra}

In this section we introduce one of the central notion
 of motivic homotopy theory, that of motivic ring spectrum.
 Our primary aim is to prove that to such an object
  is associated a Bloch-Ogus cohomology theory, a result which has not
  yet appeared in the literature of motivic homotopy theory.
 Moreover, we extend the formalism of Bloch-Ogus by proving many
  more properties, relying on some of the main constructions of motivic homotopy theory
  (\cite{Deg8}, \cite{Ayoub} and \cite{CD3}).
In the next section we will give several examples of motivic ring spectra,
among them the motivic ring spectrum representing the rigid syntomic cohomology.

We fix  a base scheme $S$ (noetherian and finite dimensional) and $\Lambda$ a
$\QQ$-algebra.
%
\subsection{Gysin morphisms and regulators} \label{sec:Gysin}

\begin{df}\label{def:mrs}
A {\em motivic ring spectrum} (over $S$) is a ring spectrum $\E$
 which is also a Morel motive. In particular it is an object of
$\DMB(S,\Lambda)$.
\end{df}
If $X$ is an $S$-scheme, we will denote by
$$
\E^{n,i}(X):=\Hom_{\DMB(S,\Lambda)}(M(X),\E(i)[n])
$$
the associated bi-graded cohomology groups.
%
\begin{rem}
\begin{enumerate}
\item In the current terminology of motivic homotopy theory,
 what we call a motivic ring spectrum should be called
 an \emph{oriented} motivic ring spectrum (see also Remark
 \ref{rem:axioms_motivic_ring_sp}).
 This abuse of terminology is justified as we
  will never consider non oriented ring spectra in
  this work.
\item In the previous section, we have seen that there exists a stronger notion
 of ring spectrum, that of stably fibrant Tate spectrum.
 The ring spectra that we will construct will always satisfies this
  stronger assumption. Moreover, given a ring spectrum in the sense of
  the above definition, it is always possible to find a stably fibrant
  Tate spectrum which is isomorphic in $\DMB(S)$ to the first given one (according to Theorem
  \ref{thm:model_ring_sp}).
 On the other hand, this stronger notion will not be used in
  this section that is why we consider above the simpler notion. 
  The stronger notion will be needed in Section \ref{sec:modules}.
\end{enumerate}

\end{rem}

\begin{num}\label{num:cycle_classes}
%
Recall that Beilinson motivic cohomology for smooth $S$-schemes
 is the cohomology represented by the unit object of $\DMB(S)=\DMB(S,\QQ)$:
$$
\HB^{n,i}(X):=\Hom_{\DMB}(M(X),\un(i)[n]).
$$
This group can also be described as the $i$-graded part
 for the $\gamma$-filtration of algebraic rational $K$-theory:
$$
\HB^{n,i}(X)=\gr_\gamma^i K_{2i-n}(X)_\QQ \ .
$$
See \cite[14.2.14]{CD3}.

By construction, the ringed cohomology $\E^{**}$
 admits a canonical action of Beilinson motivic cohomology
 $\HB^{**}$.
 Concretely, 
 for any smooth $S$-scheme $X$ and any couple of integers $(n,i)$,
 the unit map $\un \rightarrow \E$ 
 induces a canonical morphism 
\begin{equation} \label{eq:syntomic_higher_cycle}
\sigma_{\E}:\HB^{n,i}(X)=\Hom_{\DMB}(M(X),\un(i)[n])
 \rightarrow \Hom_{\DMB}(M(X),\E(i)[n])=\E^{n,i}(X)
\end{equation}
which is compatible with pullbacks and products.
This is the \emph{higher cycle class map}  
 (or equivalently the \emph{regulator}) with values in the  $\E$-cohomology. Note also that
 this map can be represented in the category $\DMt(S,\QQ)$ as a 
morphism of ring spectra:
$$
\sigma_\E:\HB \rightarrow \E\quad (\text{by abuse of notation we use the same symbol})
$$
which is unique according to \cite[14.2.16]{CD3}.

When $n=i$, it gives in particular the (usual) cycle class map:
\begin{equation} \label{eq:syntomic_cycle}
\sigma_{\E}:\CH^n(X) \rightarrow \E^{2n,n}(X)
\end{equation}
which is compatible with pullbacks and products of cycles
 as defined in \cite{Ful}.
\end{num}
\begin{num}\label{num:chern_classes}
    A motivic ring spectrum $\E$, considered as an object
      of $\DMt(S)$, is oriented
 (see Remark \ref{rem:axioms_motivic_ring_sp}).
Thus, one can apply to it the orientation theory
 of $\AA^1$-homotopy theory (see \cite{Deg12} in the arithmetic
case). 

This implies that $\E^{**}$ admits Chern classes,
 which are nothing else than the image of the Chern classes in
Chow theory through the cycle class map,
and satisfies the projective bundle formula (see \cite[2.1.9]{Deg12}).
One also gets a Chern character map in $\DMt(S,\QQ)$:
$$
\ch_\syn:KGL_\QQ \xrightarrow{\ch} \oplus_{i \in \ZZ} \HB(i)[2i]
 \xrightarrow{\sigma} \oplus_{i \in \ZZ} \E(i)[2i]
$$
where $KGL_\QQ$ is the ring spectrum representing rational algebraic K-theory
 over $R$ and $\ch$ is the isomorphism of \cite[14.2.7(3)]{CD3}.
This map induces the usual higher Chern character (see \cite{Gillet})
 for any smooth $S$-scheme $X$:
$$
\ch_n:K_n(X)_\QQ
 \rightarrow \prod_{i \in \NN} \E^{2i-n,i}(X).
$$
\end{num}

\begin{num}
Given a motivic ring spectrum $\E$, we can define a (cohomological) realization functor of $\DMB(R)$:
$$
\E(-):\DMB(R)^{op} \rightarrow \QQ_p\text{-vs},\qquad
 M \mapsto \Hom_{\DMB(S)}(M,\E).
$$
This shows that the $\E$-cohomology of a smooth $S$-scheme $X$
 inherits the functorial structure of the motive of $X$.

In particular, given a projective morphism of smooth $S$-schemes
  $f:Y \rightarrow X$, there exists a Gysin morphism on motives:
$$
M(X) \rightarrow M(Y)(-d)[-2d]
$$
where $d$ is the dimension of $f$. This was constructed in \cite{Deg8}
 and several properties of this Gysin morphism were proved there.
 Thus, after applying the functor $\E(-)$ above, one gets:
\end{num}
\begin{thm}\label{thm:Gysin}
Consider the above notations.
One can associate to $f$ a Gysin morphism in syntomic cohomology:
$$
f_*=\E(f^*):\E^{n,i}(Y) \rightarrow \E^{n-2d,i-d}(X).
$$
Moreover, one gets the following properties:
\begin{enumerate}
\item (\cite[5.14]{Deg8})
For any composable projective morphisms $f,g$, $(fg)_*=f_*g_*$.
\item (\emph{projection formula}, \cite[5.18]{Deg8})
For any projective morphism $f:Y \rightarrow X$ and
 any pair $(x,y) \in \E^{*,*}(X) \times \E^{*,*}(Y)$,
 one has:
$$
f_*(f^*(x).y)=x.f_*(y).
$$
\item (\emph{Excess intersection formula}, \cite[5.17(ii)]{Deg8})
Consider a cartesian square of smooth $S$-schemes:
$$
\xymatrix@=14pt{
Y'\ar^q[r]\ar_g[d] & X'\ar^f[d] \\
Y\ar^p[r] & X
}
$$
such that $p$ is projective. Let $\xi/Y'$ be the excess intersection
 bundle\footnote{Recall
 from \emph{loc. cit.} that one defines $\xi$ as follows:
 let us choose a closed embedding $i:Y \rightarrow P$ into a projective
 bundle over $X$ and let $Y' \rightarrow P'$ be its pullback over $X'$.
 Let $N$ (resp. $N'$) be the normal vector bundle of $Y$ in $P$
 (resp. $Y'$ in $P'$). Then, as the preceding square is cartesian,
 there is a monomorphism  $N' \rightarrow g^{-1}(N)$ 
 of vector bundles over $Y'$
  and one puts:	$\xi=g^{-1}(N)/N'$.}
associated with that square and let $e$ be its rank. \\
Then for any $y \in \E^{*,*}(Y)$, on gets:
$$
f^*p_*(y)=q_*(c_e(\xi).g^*(y)).
$$
\item For any projective morphism $f:Y \rightarrow X$,
 the following diagram is commutative:
$$
\xymatrix@=18pt{
\HB^{n,i}(Y)\ar^-{f_*}[r]\ar_{\sigma_\syn}[d]
 & \HB^{n-2d,i-d}(X)\ar^{\sigma_\syn}[d] \\
\E^{n,i}(Y)\ar^-{f_*}[r] & \E^{n-2d,i-d}(X)\ .
}
$$
\end{enumerate}
\end{thm}

\begin{rem}
\begin{itemize}
\item With the notation of Point (3) recall that 
 $\xi$ has dimension $n-m$ where $n$ (resp. $m$)
 is the dimension of $p$ (resp. $q$). In particular, 
 when the square is \emph{transverse} \emph{i.e.} $n=m$,
 one gets the more usual formula: $f^*p_*=q_*g^*$.
\item Point (2) can simply be derived from the preceding formula
 applied to the graph morphism $\gamma:Y \rightarrow Y \times_S X$
 given that $\gamma^*$ is compatible with products.
\item Point (4) shows in particular that, when $i:Z \rightarrow X$
 is a closed immersion, $i_*(1)=\sigma_\E([Z])$ is the 
 fundamental class of $Z$ in $X$. If $Z$ is a smooth divisor,
 corresponding to the line bundle $\mathcal L/X$,
 one gets in particular: 
$$
i_*(1)=c_1(\mathcal L).
$$
This property determines uniquely the Gysin morphism in the case of a closed
 immersion (see \cite{Deg8} or \cite{Panin}).

When $p:P \rightarrow X$ is the projection of a projective
 bundle of rank $n$ and canonical line bundle $\lambda$,
 one gets, again applying Point (4):
$$
p_*(c_1(\lambda)^i)=\begin{cases} 1 & \text{if i=n} \\ 0 & \text{otherwise}
\end{cases}
$$
This fact, together with the projective bundle formula in syntomic
 cohomology, determines uniquely the morphism $p_*$.

By construction,
 the Gysin morphism $f_*$ for any projective morphism $f$ is completely
 determined by the two above properties.
 \item For syntomic cohomology, Point (4) was conjectured by Besser \cite[Conjecture~4.2]{Bes:12a} (in the case of  proper morphisms) and Theorem~1.1 in {\em loc. cit.} is conditional to the conjecture. The latter result concerns the regulator of a proper and smooth surface $S$ over $R$.
We also note that Point (4) has already been used (in the projective morphism case, although stated for proper maps) in \cite[p. 505]{Lan:11a} but the references given there is a draft of \cite{ChiCicMaz:Cyc10} which turns to be different from the published version and does not contain the above statement, neither its proof.
\end{itemize}
\end{rem}

\begin{ex}
Let $f:Y \rightarrow X$ be a finite morphism between
 smooth connected $S$-schemes. Let $d$ be the degree
 of the extension of the corresponding function fields.
Then one gets the \emph{degree formula} in $\E$-cohomology:
 for any $x \in \E^{*,*}(X)$,
$$
f_*f^*(x)=d.x.
$$
Indeed, according to \ref{thm:Gysin}(1),
$$
f_*f^*(x)=f_*(1.f^*(x))=f_*(1).x.
$$
Then one gets $f_*(1)=d$ from \ref{thm:Gysin}(4)
 and the degree formula in Beilinson motivic cohomology.
\end{ex}

As a corollary of Point (4) of the preceding theorem,
 one obtains the Riemann-Roch formula in $\E$-cohomology.
\begin{cor}
Let $f:Y \rightarrow X$ be a projective morphism between
 smooth $S$-schemes. Let $\tau_f$ be the virtual tangent bundle of
 $f$ in $K_0(X)$: $\tau_f=[T_X]-[T_Y]$, the difference
 of the tangent bundle of $X/S$ with that of $Y/S$.
Then for any element $y \in K_n(Y)_\QQ$,
 one gets the following formula:
$$
\ch_\E\big(f_*(y)\big)=f_*\big(\td(\tau_f).\ch_\E(y)\big)
$$
where $\td(\tau_f)$ is the Todd class of the virtual vector bundle
 $\tau_f$ in  $\E$-cohomology (defined for example as the image of the
 usual Todd class in Chow groups by the cycle class map).
\end{cor}
In fact, this corollary is deduced from the Riemann-Roch formula
 in motivic cohomology after applying to it the  higher cycle class
 and applying Point (4) of the previous theorem.
\subsection{The six functors formalism and Bloch-Ogus axioms}
\label{sec:Bloch-Ogus}

In this section, we will recall some consequences 
 of Grothendieck six functors formalism
 established for Beilinson motives
 (see \cite[2.4.50]{CD3} for a summary)
 and apply this theory to the spectra considered in this paper.
We will  consider only separated $S$-schemes
 of finite type over $S$.
 We will also consider an abstract object $\E$ of $\DMB(S)$.

\begin{num}\label{num:4theories}
We associate with $\E$
 four homology/cohomology theories defined 
 for an $S$-scheme $X$ with structural morphism $f$ 
 and a pair of integers $(n,i)$ as follows:
\begin{center}
\begin{tabular}{|l|l|}
\hline
Cohomology & $\E^{n,i}(X)=\Hom(\un_S,f_*f^*\E(i)[n])$ \\
\hline
Homology & $\E_{n,i}(X)=\Hom(\un_S,f_!f^!\E(-i)[-n])$ \\
\hline
Cohomology with compact support & $\E^{n,i}_c(X)=\Hom(\un_S,f_!f^*\E(i)[n])$ \\
\hline
Borel-Moore homology & $\E_{n,i}^\BM(X)=\Hom(\un_S,f_*f^!\E(-i)[-n])$ \\
\hline
\end{tabular}
\end{center}
We will use the terminology \emph{c-cohomology} (resp. \emph{BM-homology})
 for cohomology with compact support (resp. Borel-Moore homology).

Note that these definitions, applied to the unit object $\un$ of $\DMB(S)$,
 yield the four corresponding motivic theories.
Also, these definitions are (covariantly) functorial in $\E$.
In particular, if $\E$ admits a structure of a monoid in $\DMB(S)$
 (\emph{i.e.} $\E$ is a ring spectrum), the unit map 
 $\eta:\un \rightarrow \E$ yields regulators in all four theories.

When $X/S$ is proper, as $f_*=f_!$, one gets identifications:
$$
\E^{n,i}(X)=\E^{n,i}_c(X), \qquad \E_{n,i}^\BM(X)=\E_{n,i}(X).
$$
\end{num}
\begin{num}\label{num:functo}
\textit{Functoriality properties}.--
We consider a morphism of $S$-schemes:
$$
\xymatrix@=10pt{
Y\ar^f[rr]\ar_q[rd] && X\ar^p[ld] \\
& S &
}
$$
Using the adjunction map
 $ad_f:1 \rightarrow f_*f^*$ (resp. $ad'_f:f_!f^! \rightarrow 1$),
 we immediately obtain that cohomology is contravariant
 (resp. homology is covariant) by 
 composing on the left by $p_*$ (resp. $p_!$)
 and on the right by $p^*$ (resp. $p^!$).

When $f$ is proper, $f_!=f_*$.
 Using again $ad_f$, $ad'_f$, one deduces that
 c-cohomology 
 (resp. BM-homology) is contravariant 
 (resp. contravariant) with respect to proper maps.

When $f$ is smooth of relative dimension $d$,
 one has the relative purity isomorphism:
$$f^! \simeq f^*(d)[2d]$$
  (see in \cite{CD3}: Th. 2.4.50 for the statement
	and Sec. 2.4 for details on relative purity).
 In particular, one derives from $ad_f$ and $ad'_f$ the following
 maps:
$$
f_*:\E^{n,i}_c(X) \rightarrow \E^{n-2d,i-d}_c(Y), \quad
 f^*:\E_{n,i}^\BM(X) \rightarrow \E_{n+2d,i+d}^\BM(Y).
$$
Finally, when $f$ is proper and smooth of relative dimension $d$
 one gets in addition:
$$
f_*:\E^{n,i}(X) \rightarrow \E^{n-2d,i-d}(Y), \quad
 f^*:\E_{n,i}(X) \rightarrow \E_{n+2d,i+d}(Y).
$$
Let us summarize the situation:
\begin{center}
\begin{tabular}{|l|c|c|}
\hline
theory & covariance (degree) & contravariance (degree) \\
\hline \hline
Cohomology & smooth proper, (-2d,-d) & any \\
\hline
Homology & any & smooth proper, (+2d,+d) \\
\hline
Cohomology with compact support
 & smooth, (-2d,-d) & proper \\
\hline
Borel-Moore homology & proper & smooth, (+2d,+d) \\
\hline
\end{tabular}
\end{center}
\end{num}

\begin{rem} \label{rem:projection_formulas1}
The fact that the functorialities constructed above
 are compatible with composition is obvious except
 when a smooth morphism is involved. This last case follows 
 from the functoriality of the relative purity isomorphism
 proved by Ayoub in \cite{Ayoub}.

When considering one of the four theories associated with $\E$,
 one can mix the two kinds of functoriality in a projection formula
 as usual. In fact, given a cartesian square:
$$
\xymatrix@=14pt{
Y'\ar^g[r]\ar_q[d] & X'\ar^p[d] \\
Y\ar^f[r] & X
}
$$
such that $f$ is proper and smooth (or $f$ smooth and $g$ proper when
 considering $\E_c$ or $\E^\BM$),
 one obtains respectively:
\begin{itemize}
\item $f^*p_*=q_*g^*$ for the two homologies,
\item $p^*f_*=g_*q^*$ for the two cohomologies.
\end{itemize}
This is a lengthy check coming back to the definition of the relative
 purity isomorphism. The essential fact is that 
$$
g^{-1}(T_{Y/X})=T_{Y'/X'}
$$
 where $T_{Y/X}$ (resp. $T_{Y'/X'}$) is the tangent bundle of $f$ (resp. $g$).
\end{rem}

%
%

\begin{num} \textit{Products}.--
Let us now assume that $\E$ is a ring spectrum,
 with unit map $\eta:\un_S \rightarrow \E$
 and product map $\mu:\E \otimes \E \rightarrow \E$.

Of course, for any $S$-scheme $X$ with structural map $f$,
 we can define a product on cohomology, 
 sometimes called the cup-product:
$$
\E^{n,i}(X) \otimes \E^{m,j}(X) \rightarrow \E^{n+m,i+j}(X),
 (x,y) \mapsto xy=x \scup y ;
$$
given cohomology classes
$$
x:\un_X \rightarrow f^*\E(i)[n], y:\un_X \rightarrow f^*\E(j)[m],
$$
we define $xy$ as the following composite map:
$$
\un_X \xrightarrow{x \otimes y} f^*(\E)(i)[n] \otimes f^*(\E)(j)[m]
 =f^*(\E \otimes \E)(i+j)[n+m]
 \xrightarrow{\mu} f^*(\E)(i+j)[n+m].
$$
This product is obviously commutative and associative.
Note one can also define an exterior product on cohomology as follows:
$$
\E^{n,i}(X) \otimes \E^{m,j}(Y) \rightarrow \E^{n+m,i+j}(X \times_S Y),
 (x,y) \mapsto p_1^*(x).p_2^*(y)
$$
where $p_1$ (resp. $p_2$) is the projection $X \times_S Y/X$
 (resp. $X \times_S Y/Y$).

One can also define \emph{exterior products}
 on c-cohomology.
 Consider a cartesian square:
$$
\xymatrix@=14pt{
X \times_S Y\ar^-{f'}[r]\ar_{g'}[d]\ar|h[rd] & Y\ar^g[d] \\
X\ar_f[r] & S
}
$$
of separated morphisms of finite type.
We define the following product on c-cohomology:
$$
\E^{n,i}_c(X) \otimes \E^{m,j}_c(Y)
 \rightarrow \E^{n+m,i+j}_c(X \times_S Y), (x,y) \mapsto x \times y
$$
which associates to any maps
$$
x:\un_S \rightarrow f_!f^*\E(i)[n], y:\un_S \rightarrow g_!g^*\E(j)[m],
$$
the following composite map $x \times y$:
\begin{align*}
\un_S \xrightarrow{x \otimes y} &f_!f^*(\E)(i)[n] \otimes g_!g^*(\E)(j)[m] \\
& \simeq f_!\big(f^*(\E)(i)[n] \otimes f^*g_!g^*(\E)\big)(i+j)[n+m] \\
 & \simeq f_!\big(f^*(\E)(i)[n] \otimes g'_!f^{\prime*}g^*(\E)\big)(i+j)[n+m] \\
 & \simeq f_!g'_!\big(g^{\prime*}f^*(\E)(i)[n] \otimes f^{\prime*}g^*(\E)\big)(i+j)[n+m] \\
&=h_!h^*(\E \otimes \E)(i+j)[n+m]
 \xrightarrow{\mu} h_!h^*(\E)(i+j)[n+m]
\end{align*}
where the first and third isomorphisms follow from the projection formula
 \cite[2.4.50(v)]{CD3} and the second one from the exchange isomorphism
 \cite[2.4.50(iv)]{CD3}.

One can check the following formulas:
$$
(x \times y) \times z=x \times (y \times z), 
 x \times y=y \times x
$$
through the respective isomorphisms
$$
(X \times_S Y) \times_S Z \simeq (X \times_S Y) \times_S Z,
 X \times_S Y\simeq Y \times_S X.
$$

Further, because c-cohomology is contravariant with respect
 to proper morphism, given any $S$-schemes $X$ (separated of finite type),
 the diagonal embedding $\delta:X \rightarrow X \times_S X$
 allows to define an inner product on c-cohomology:
$$
\E^{n,i}_c(X) \otimes \E^{m,j}_c(X)
 \rightarrow \E^{n+m,i+j}_c(X), (x,x') \mapsto \delta^*(x \times x').
$$
When $X/S$ is proper, one can check this product coincides with
 cup-product on cohomology.
\end{num}

\begin{rem}
Let $f:Y \rightarrow X$ be a proper smooth morphism.
According to the projection formulas established in Remark~\ref{rem:projection_formulas1}, 
one can check that for any couple $(y,x)$
 either in $\E^{n,i}(Y) \times \E^{m,j}(X)$
 or in $\E^{n,i}_c(Y) \times \E^{m,j}_c(X)$,
one gets the following usual projection formula (for products):
$$
f_*(x.f^*(y))=f_*(x).y.
$$
In fact, in each case, one uses the relevant formula of Remark
 \ref{rem:projection_formulas1}, 
 the external product and the following formulas:
$$
y \times f^*(x)=(1_Y \times_S f)^*(y \times x),
 f_*(y) \times x=(f \times_S 1_X)_*(y \times x).
$$
\end{rem}

\begin{num} \textit{Cap product}.--
One can extend the cohomology theory associated with $E$
 to a theory with support. Given any closed immersion
 of $S$-schemes:
$$
\xymatrix@R=10pt@C=24pt{
Z\ar_g[rd]\ar^i[rr] && X,\ar^f[ld] \\
& S &
}
$$
one puts:
$$
\E^{n,i}_Z(X)=\Hom(i_*(\un_Z),f^*\E(i)[n])
 =\Hom(\un_Z,i^!f^*\E(i)[n]).
$$
This theory satisfies all the usual properties.
 We refer the reader to \cite[\textsection 1.2]{Deg12}
 for a detailed account.

Assuming again $\E$ is a ring spectrum with product map
 $\mu:\E \otimes \E \rightarrow E$,
 one defines, following Bloch and Ogus,\cite{BO}, 
 the \emph{cap-product with supports}:
$$
\E_{n,i}^\BM(X) \otimes \E^{m,j}_Z(X)
 \rightarrow \E_{n-m,i-j}^\BM(Z), (x,z) \mapsto x \scap z.
$$
Let us first introduce classical pairing of functors
 (see \cite[IV, \textsection 1.2]{SGA4demi}):
 given any objects $A$ and $B$ of $\DMB(S)$, one considers
 the following composite map
$$
f_!(f^!(A) \otimes f^*(B))
 \xrightarrow{Ex} \lbrack f_!f^!(A)\rbrack \otimes B
 \xrightarrow{ad'_f} A \otimes B
$$
where the first map is the isomorphism of the projection
 formula (\cite[2.4.50]{CD3}) and the second one
 is the counit of the adjunction $(f_!,f^!)$.
 One thus deduces by adjunction the following pairing
$$
f^!(A) \otimes f^*(B) \xrightarrow{\eta_f} f^!(A \otimes B).
$$
Thus, given maps
$$
x:\un_X \rightarrow f^!(\E), \ z:i_*(\un_Z) \rightarrow f^*(\E)
$$
one defines $x \scap z$ from the following composite map:
$$
i_*(\un_Z) \xrightarrow{x \otimes z} 
 f^!(\E) \otimes f^*(\E) \xrightarrow{\eta_f} f^!(\E \otimes \E)
 \xrightarrow{\mu} f^!(\E)
$$
using $i_*=i_!$, the adjunction $(i_!,i^!)$ and $i^!f^!=g^!$.
\end{num}

\begin{rem}
Consider a cartesian square of $S$-schemes
$$
\xymatrix@=20pt{
T\ar^{k}[r]\ar_g[d] & Y\ar^f[d] \\
Z\ar^i[r] & X
}
$$
such that $i$ is a closed immersion and $f$ is proper.
Then, for any couple $(y,z) \in \E_{n,i}^\BM(X) \otimes \E^{m,j}_Z(X)$,
 one obtains the following formula:
$$
f_*(y) \scap z=g_*(y \scap f^*(z)).
$$
\end{rem}

\begin{num}
Suppose again $\E$ is a ring spectrum with unit map
 $\eta:\un_S \rightarrow \E$.

Let $f:X \rightarrow S$ be a smooth $S$-scheme of relative
 dimension $d$. Then, according to \cite[2.4.50(iii)]{CD3},
 one obtains a canonical isomorphism of functors:
$$
\mathfrak p_f:f^! \rightarrow f^*(d)[2d].
$$
In particular, one gets a canonical map
$$
\eta_X:\un_X=f^*(\un_S) \xrightarrow{f^*(\eta)} f^*(\E)
 \xrightarrow{\mathfrak p_f^{-1}} f^!(\E)(-d)[-2d]
$$
which corresponds to a homological class
 $\eta_X \in \E_{2d,d}^\BM(X)$.
The following result is now a tautology:
\end{num}
\begin{prop}
Consider the above assumptions, and let $Z \subset X$
 be any closed subset.
Then the following map:
$$
\E_Z^{n,i}(X) \rightarrow \E^\BM_{2d-n,i-n}(Z),
 z \mapsto \eta_X \scap z
$$
is an isomorphism.
\end{prop}

One can now summarize some of the main properties
 we have proved so far as follows:
\begin{cor}
The couple of functors $\big(\E^{**}$, $\E^\BM_{**}\big)$
 form a Poincar\'e duality theory with supports
 in the sense of Bloch and Ogus (\cite[Def 1.3]{BO}).
\end{cor}
This is the case in particular for syntomic
 cohomology and syntomic BM-homology.

\begin{num} \label{num:hdescent}
\textit{Descent theory}.--
Recall (see \cite[\textsection 3.1]{CD3})
 that a diagram of $S$-schemes $(\mathcal X,I)$
 is the data of a small category $I$ and a functor
 $\mathcal X:I \rightarrow \mathscr S$.
A morphism of diagrams
 $\varphi=(\alpha,f):(\mathcal X,I) \rightarrow (\mathcal Y,J)$
 is the data of a functor $f:I \rightarrow J$
 and a natural transformation
 $\alpha:\mathcal X \rightarrow f^*(\mathcal Y)$
 where $f^*(\mathcal Y)=\mathcal Y \circ f$.

Accorded to \cite[\textsection 3.1]{CD3},
 the fibered triangulated category $\DMB$ can be
 extended to the category of diagrams.
 Moreover, for any morphism of diagrams
 $\varphi:(\mathcal X,I) \rightarrow (\mathcal Y,J)$,
 one has an adjoint pair of functors:
$$
\varphi^*:\DMB(\mathcal Y,J)
 \leftrightarrows \DMB(\mathcal X,I):\varphi_*.
$$
Consider a diagram of $S$-schemes $(\mathcal X,I)$
 and the canonical morphism
 $\varphi:(\mathcal X,I) \rightarrow (S,*)$
 where $*$ is the final category.
Then one defines the cohomology of $(\mathcal X,I)$ as:
$$
\E^{n,i}(\mathcal X,I)=\Hom(\un,\varphi_*\varphi^*(\E)(i)[n]).
$$
This is contravariant with respect to morphisms of diagrams.

In particular one has extended the cohomology $\E^{*,*}$
 to simplicial $S$-schemes.
 The h-topology was introduced by Voevodsky in \cite{Voe1}.
 Recall that a h-cover $f:Y \rightarrow X$ of $S$-schemes
 is a universal topological epimorphism (e.g. faithfully flat maps,
 proper surjective maps).
 Then the h-descent theorem for Beilinson motives
 (\cite[14.3.4]{CD3}) states the following:

For any quasi-excellent $S$-scheme $X$ and any hypercover
 $p:\mathcal X \rightarrow X$ for the h-topology,
 the canonical map
$$
p^*:\E^{n,i}(X) \rightarrow \E^{n,i}(\mathcal X)
$$
is an isomorphism. In particular, one gets the usual 
 spectral sequence:
$$
E_1^{p,q}=\E^{p,i}(\mathcal X_q) \Rightarrow \E^{p+q,i}(X).
$$
\end{num}

\begin{rem}
As already remarked in \cite{CD1},
 the preceding descent theory together with 
 De Jong resolution of singularities,
 shows that in the case where $S$ is the spectrum of a field
 (non necessarily perfect),
 the cohomology $\E^{*,*}$ is uniquely determined by its restriction 
 to smooth schemes.
\end{rem}

\section{Syntomic spectrum}\label{sec:absdR}
In this section we construct several motivic  ring spectra  (see Def.~\ref{def:mrs}):  $\E_\FdR,\E_\rig, \E_\phi, \E_\syn$.
First for a field
$K$ of characteristic zero we construct $\E_\FdR$ representing the filtered part
of the de Rham cohomology of a $K$-scheme. i.e.
$$
\E_\FdR^{n,i}(X):=\Hom_{\DMt(\eta,\QQ)}(\sus \QQ(X),\E_\FdR(i)[n]) \simeq
F^iH^n_\dR(X)\ .
$$
Then we define $\E_\rig$ which represents the rigid cohomology of Berthelot.
This was already proved in \cite{CD2} in a different way. 
For both $\E_\FdR$ and $\E_\rig$ we use the criteria of
Proposition~\ref{prop:ring_spectrum_existence}.

Finally we get a motivic ring spectrum $\E_\syn$ for the rigid syntomic
cohomology as a homotopy limit of a diagram of ring spectra.
\subsection{Cosimplicial tools}
\begin{num}
  Let $\Delta$ be the category of finite ordered sets $[n]:=\{0,...,n\}$ as
objects and monotone nondecreasing functions as morphisms. Let
$\delta_i(n):[n-1]\to [n]$ (resp. $\sigma_i(n):[n]\to [n-1]$) be the
usual\footnote{\textit{i.e.} the image of $\delta_i(n)$
 is $[n]\setminus \{i\}$.}
(co)face (resp. (co)degeneracy) map. In case there is no
ambiguity we will simply write $\delta_i,\sigma_i$. Given a category $C$, a
\emph{simplicial} (resp. \emph{cosimplicial}) object of $C$ is a functor from
$\Delta^\circ$ (resp. $\Delta$) to $C$. 

For instance let $A_n=\QQ[T_0,...,T_n]/(\sum T_i-1 )$, then this is a
simplicial $\QQ$-algebra in an obvious way. It follows that the associated   differential graded algebra (dga)
of K\"ahler differentials
\begin{equation}
 \omega_n:= \Omega^\bullet_{A_n/\QQ} \quad n\ge 0 
\end{equation}
is a simplicial dga over $\QQ$. We will denote by $\delta^i=\delta_i^*$ (resp.
$\sigma^i=\sigma_i^*$) the structural morphisms.

Now let $M$ be a cosimplicial abelian group and $sM$ the associated simple complex ($sM^i=M[i]$ and the differentials are the alternate sums of the coface morphisms).  Its \emph{standard normalization} $NM$ is  the subcomplex of  $sM$ s.t. $N^qM:=\bigcap_i\ker(\sigma_i)\subset M^q$. Then inclusion $NM\to sM$ is a homotopy equivalence. Now if $M$ is also a cosimplicial commutative monoid the Alexander-Whitney product\footnote{This
is  given as follows. Let
$\delta^-:[q]\to [q+q']$ (resp. $\delta^+:[q']\to [q]$) the map with image
$\{0,1,...,q\}$  (resp. $\{q,q+1,...,q+q'\}$), then define 
$a*b:=\delta^-(a)\cdot \delta^+(q)$.} gives a (differential graded) monoid structure on $sM$ and $NM$, but this is not necessarily (graded) commutative.  Thus we consider the following construction due to Thom and Sullivan. Let $M$ be a cosimplicial dga, we define
$$
\widetilde{N}^qM\subset \prod_m \omega_m^q\otimes M^m
$$
as the submodule whose elements are sequences $(x_m)_{m\ge 0}$ such that
$$
(\id \otimes \delta_i)x_m=(\delta^i\otimes \id)x_{m+1}\quad ,\ (\sigma^i\otimes
\id)x_m= (\id\otimes\sigma_i)x_{m+1}
$$
and define the differentials $D:\widetilde{N}^qM\to \widetilde{N}^{q+1}M$ by
$D=((-1)^q\id\otimes d)+\id\otimes \partial)$, where $d$ (resp. $\partial$) is the differential of $M$ (resp. $\omega_m$).
With the above notation if $M$ is further a cosimplicial commutative monoid then
$\widetilde{N} M$ is a commutative monoid too. Namely we can define
\begin{equation}\label{eq:TSproduct}
		\widetilde{N} M \otimes \widetilde{N} M \rightarrow
\widetilde{N} M
\end{equation}
induced by $(\alpha \otimes m)\otimes (\alpha'\otimes m')=\alpha\wedge
\alpha'\otimes (m\cdot m')$.

Moreover the complex $\widetilde{N} M$ is quasi-isomorphic to the standard
normalization $NM$ (and then to $sM$).\footnote{The isomorphism is induced by
the integration map $\int:\omega^\bullet_n \otimes M^n \to \QQ[-n]\otimes M^n$
defined by
\[
	(dT_1\wedge\cdots\wedge dT_n)\otimes m\mapsto \frac{1}{n!}\otimes m\ .
\] 
}

We can  extend  the above constructions to the setting of cosimplicial dg
abelian groups. Given such an $M=M^{pq}$ (where $q$ is the cosimplicial
parameter), then $sM$ (resp. $NM$, $\widetilde{N} M$) is naturally a double complex
and we can apply the total complex functor, denoted by $\tot$,
 to obtain a dg abelian group.

Now we are ready to state a technical result well known to the specialists.
\end{num}
\begin{prop}[cf. {\cite{HinSch:87a},\cite[Appendix]{HY99}}] \label{prop:thomsullivan}
	Let $M$ be a cosimplicial (commutative) dga over $\QQ$ then there exists
a canonical (commutative) dga $\widetilde{N} M$ and a quasi-isomorphism
$\int: \widetilde{N} M \to sM$ inducing an isomorphism of (commutative) dg algebras
in cohomology $H(\int):H(\widetilde{N} M)\to H(sM)$.
\end{prop}
\begin{num}(Godement resolutions)
	Let $u:P\to X$ be a morphism of Grothendieck sites and let 
$P^{\sim}$ (resp. $X^{\sim}$) be
 the category of abelian sheaves on $P$ (resp. $X$). Then we have a 
	 pair of adjoint functors $(u^*,u_*)$, where $u^*:X^{\sim} \to
P^{\sim}$, $u_*:P^{\sim} \to X^{\sim}$.
 For any object $\mathcal{F}$ of $X^{\sim}$
	we can define a co-simplicial object $B^*(\mathcal{F})$ whose component
in
 degree $n$ is   $(u_*u^*)^{n+1}(\mathcal{F})$.\footnote{The cosimplicial
structure is defined as follows. First  let $\eta:\id_{X^{\sim}}\to u_*u^*$ and
$\epsilon:u^*u_*\to \id_{P^{\sim}}$ be the natural transformations induced by
adjunction.

	Endow $B^{n}(\mathcal{F}):=(u_*u^*)^{n+1}(\mathcal{F})$ with co-degeneracy maps
	\[
		\sigma_i^n:=(u_*u^*)^i u_*\epsilon u^* (u_*u^*)^{n-1-i}:B^{n}(\mathcal{F})\to B^{n-1}(\mathcal{F})\quad i=0,..., n-1
	\]
	and co-faces
	\[
		\delta^{n-1}_i:=(u_*u^*)^i\eta(u_*u^*)^{n-i}:B^{n-1}(\mathcal{F})\to B^{n}(\mathcal{F})\quad i=0,...,n\ .
	\]}
\end{num}
\begin{prop}\label{prop:godement} Let $u:P\to X$ a morphism of sites and $\mathcal{F}$ a complex of sheaves on $X$. If    $u^*$ is exact and conservative,  then
	\begin{enumerate}
		\item The complex $\Gdm_P(\mathcal{F}):=s B^*(\mathcal{F})$ is a functorial flask resolution of  $\mathcal{F}$
		\item If  $\mathcal{F}$ is a $\QQ$-linear sheaf,  the
Thom-Sullivan normalization $\tGdm_P(\mathcal{F}):=
\widetilde{N}B^*(\mathcal{F})$ is a functorial resolution of
$\mathcal{F}$; 
		\item If  $\mathcal{F}$ is a sheaf of  (commutative) dga over $\QQ$, then the complex $\tGdm_P(\mathcal{F})= \widetilde{N}B^*(\mathcal{F})$ is a sheaf of commutative dga and  the canonical isomorphism $	H^*(X,\mathcal{F})\cong H^*(\Gamma(X\tGdm\mathcal{F}))$ is compatible with respect to the multiplicative structure. 
	\end{enumerate}
\end{prop}
\begin{proof}
	Since $u^*$ is exact and conservative,    to show that the canonical map $b_{\mathcal{F}}:\mathcal{F}\to sB^*(\mathcal{F})$  is a quasi-isomorphism is sufficient to prove that   $u^*b_{\mathcal{F}}$ is a quasi-isomorphism. This follows from the fact that  the augmented complex 
	\[
	   u^*\mathcal{F}\to u^*B^0(\mathcal{F})\to u^*B^1(\mathcal{F})\to \cdots
	\]
	is null-homotopic: the homotopy $h^i:u^* (u_*u^*)^i(\mathcal{F})\to u^*(u_*u^*)^{i-1}(\mathcal{F})$ is induced by the counit $u^*u_*\to \id$ and one checks easily that $\id= d^{i-1}\circ h^{i}+h^{i+1}\circ d^i$, where $d^i$ is given by the alternating sum of cofaces. 
	  The rest follows directly from Prop.~\ref{prop:thomsullivan} and the existence of
 a family of canonical maps
	\[
		\cup_n:B^n(F)\otimes B^n(G)\rightarrow B^n(F\otimes G)
	\]
	compatible with the cosimplicial structure. We leave to the reader to check
that if   $\mathcal{F}^*$  is further  a (commutative) dga on $X^{\sim}$ then
$B^*(\mathcal{F}^*)$ is a cosimplicial (commutative) dga\footnote{In
fact one needs to take care of the signs: \[
		\cup_{n}^{ab}:B^n(\mathcal{F}^a)\otimes B^n(\mathcal{F}^b)\to B^n(\mathcal{F}^a\otimes B^n(\mathcal{F}^b))\ , \ \cup_{n}^{ab}=(-1)^{na}\cup_n\ .
	\]

	}.
\end{proof}
\begin{num}[Enough points]\label{num:enough_points}
	We will use the above construction in the case $X$ is the site
associated to a scheme or
a dagger space (in the case of a dagger space we take the site associated to its
$G$-topology). In both cases we let $P$  be the category $Pt(X)$ of
site-theoretical points  of  $X$. For a general $X$   the canonical map
$u:Pt(X)\to X$ is not conservative. The latter property is guaranteed in the two
cases we are interested in. It suffices to exhibit a subcategory $C$ of $Pt(X)$
(with the discrete topology) such that  $u$ restricted to $C$ is conservative.
When $X$ is associated to a scheme (resp. a dagger space) we let $C$ be the
category of  its Zariski points (resp. its Berkovich or adic points). This is
enough as explained in  \cite[\S~3]{ChiCicMaz:Cyc10} or \cite[\S~3]{Tam:Kar11}. 
	
	From now on we will simply write $\tGdm$ instead of $\tGdm_{Pt(X)}$ with $X$ as above.
\end{num}
\subsection{De Rham cohomology}
\begin{num}[The Hodge Filtration]
We recall some well known facts about algebraic de Rham cohomology (see for
instance \cite{Jan:Mix90}). Let $K$ be a field of characteristic zero and
$X$ be a smooth and algebraic $K$-scheme.
Fix a compactification $g:X\to \bar X$ such that the complement $D=\bar
	X \setminus X$ is a normal crossing divisor\footnote{Such a compactification exists by  the Nagata's compactification theorem  and the result of Hironaka on the resolution of singularities.} . Then  consider the complex $\Omega^\bullet_{\bar X /
	K}\langle D\rangle$ of differential forms on $\bar X$ with logarithmic
	differential poles along $D$. The natural inclusion 
		$\Omega^\bullet_{\bar X /	K}\langle D\rangle
		\subset g_*\Omega_{X/K}^\bullet$ is a quasi-isomorphism and we
	define the Hodge filtration on the de Rham cohomology of $ X$ by
	$$
		F^iH_{\dR}^n(X/K):=H^n(\bar X,F^i\Omega^{\bullet}_{\bar X / K}\langle
	D\rangle)
	$$
	where  $F^i \Omega^{\bullet}_{\bar X / K}\langle D\rangle$
is the stupid filtration.

A remarkable result of Deligne says that (for $K=\CC$) the Hodge filtration
does not depend on the chosen compactification. Moreover given a morphism
$f:X\to Y$ of smooth algebraic schemes over $\CC$ the induced morphism
on de Rham cohomology is strictly compatible w.r.t the Hodge
filtrations\footnote{A morphism $f:A\to B$ of filtered vector spaces is strict
if $f(F^iA)=f(A)\cap F^iB$.}. Then the same holds for $H_{\dR}^n(X/K)$ where
$K\subset \CC$ is a field of characteristic zero.
\end{num}
\begin{prop} Let $X$ be a smooth $K$-scheme.
	\begin{enumerate}
		\item For any normal crossing compactification $\bar X$ of $X$
the resolution $\tGdm(\Omega^{\bullet}_{\bar X / K}\langle D\rangle)$ (notation as in
\S~\ref{num:enough_points}) gives a sheaf of filtered commutative
dga\footnote{Set $F^i\tGdm=\tGdm F^i $.}
and $F^iH_{\dR}^n(X/K)\cong
H^n(\Gamma({\bar X},\tGdm(F^i\Omega^{\bullet}_{\bar X / K}\langle D\rangle))$.
		\item The following complexes  
		\begin{eqnarray}
		    E_{\FdR,i}(X):= & \colim_{\bar X} \Gamma({\bar
X},\tGdm(F^i\Omega^{\bullet}_{\bar X / K}\langle D\rangle)\\
E'_{\dR}(X):= & \colim_{\bar
X}\Gamma({\bar X},\tGdm(   g_*\Omega^\bullet_X))\\
E_{\dR}(X):=& \Gamma( X,\tGdm( \Omega^\bullet_X))
		\end{eqnarray}
		are functorial in $X$ and there are functorial quasi-isomorphisms\footnote{We introduce $E'_{\dR}$ since there is no natural map between $E_\dR$  and $E_{\FdR,i}$.}
		\[
			E_{\FdR,0}(X)\rightarrow
			E'_{\dR}(X)\leftarrow E_{\dR}(X)
		\]
	\end{enumerate}
\end{prop}
\begin{proof}
	By definition $\Omega^{\bullet}_{\bar X / K}\langle D\rangle$ is a commutative (filtered) dga. Let
	\[
	F^i\tGdm( \Omega^{\bullet}_{\bar X / K}\langle D\rangle) =\tGdm   (F^i \Omega^{\bullet}_{\bar X / K}\langle D\rangle)	\ .
	\]
	Then  $ \tGdm( \Omega^{\bullet}_{\bar X / K}\langle
D\rangle)$  is  a (sheaf of) filtered commutative dga by
Proposition~\ref{prop:godement}. This concludes the proof of point (1).
	
As the complex of sheaves $\Omega^{\bullet}_{\bar X / K}\langle
D\rangle$ is
functorial\footnote{Morphisms of pairs are morphisms of commutative squares.}
with respect to the pair $(X,D)$, the same is true for $F^i\tGdm( \Omega^{\bullet}_{\bar X / K}\langle D\rangle)$.  Note  that the  category of normal crossing
compactifications is filtered. Hence the above colimit is  quasi-isomorphic to any of its
elements. What remains to prove follows directly from the definitions.
\end{proof}
\begin{ex}\label{ex:dlog}
Let $X=\PP^1_K\setminus \{0,\infty\}$.  By construction
$E_{\FdR,1}(X)$ is a complex starting in degree $1$.
Let $\dlog\in\Gamma(\PP^1_K, \Omega^1_{\PP^1_K}\langle
0,\infty\rangle)=H^0(E_{\FdR,1}(X)[1])=H^1(E_{\FdR,1}(X))$ be  the section
defined by $dT/T$, for a local parameter $T$ at $0$. Note that
the class of $d\log$ is a generator for $F^1
H^1_{\dR}({X})\cong K$. We will denote it by $c_1^{\FdR}$.
\end{ex}
\begin{prop}
	There exists a motivic ring spectrum $\E_{\FdR}$ whose components are
the complexes $E_{\FdR,i}$ and such that
	\[
		F^iH^n_{\dR}(X)=\Hom_{\DMt(K,\QQ)}(\un,\E_{\FdR}(i)[n])\ .
	\]
\end{prop}
\begin{proof}
By the previous Lemma the family $E_{\FdR,i}$ forms a  $\NN$-graded commutative
monoid. The $\dlog$ of the above example gives a morphism $\QQ(\mathbb{G}_{m,K})
\to E_{\FdR,1}$.
	According to Proposition~\ref{prop:ring_spectrum_existence} we have to
prove the following.

	(Excision and homotopy) $E_{\FdR,i}$ is both Nis-local and
$\AA^1$-local. We know that
$E_{\dR}$ is Nis$/\AA^1$-local so that the same holds for $E_{\FdR,0}$. The same holds for
$E_{\FdR,i}$ since the canonical maps $E_{\FdR,i}\to E_{\FdR,0} $ induce the
Hodge
filtration on cohomology. Then thanks to the strictness it is easy to conclude
(see also the paragraph following this proof).
	
	(Stability) The cup product with $\dlog=dT/T$ induces an isomorphism
	$$
		H^{n}( E_{i}(X)) \cong H^{n+1}(E_{i+1}(\GG\times
	X))/H^{n+1}(E_{i+1}(X))\ .
	$$
	Let $g:X\to \bar X$ be a normal crossing compactification with
complement $D$. Then $\GG\times X\to \PP^1\times \bar X$ is a
normal crossing compactification with complement $E=\{0,\infty\}\times \bar X
\cup \PP^1\times D$.
	We have  to prove that $\Omega_{\PP^1\times \bar X}\langle
E\rangle=p_1^*\Omega_{\PP^1}\langle 0,\infty\rangle\otimes p_2^*\Omega_{\bar
X}\langle D\rangle$. This can be checked locally by choosing \'etale
coordinates.
Then it is easy
to prove the filtered
K\"unneth decomposition $F^{i+1}H^{n+1}_{\dR}(\GG\times X)= H^0_{\dR}(\GG)\otimes
F^{i+1}H_{\dR}^{n+1}(X) \oplus H^1_{\dR}(\GG)\otimes F^{i}H_{\dR}^{n}(X)$ since $
F^jH^j_{\dR}(\GG)= H^j_{\dR}(\GG)\cong K$ for $j=0,1$. As $H^j_{\dR}(\GG)=Kd\log$
the claim is proved.

(Orientation) This is obvious: the morphism of $\AA^1\setminus\{0\}$ induced by $T\mapsto 1/T$
sends $dT/T$ to 
 $-dT/T$
 as an
element of $H^0(\PP^1_K, \Omega^1_{\PP^1_K}\langle 0,\infty\rangle)\subset
E_{\FdR,1}(\AA^1\setminus \{0\})$.
\end{proof}	
%
%
\begin{num}[Variation on dagger spaces]\label{num:dagger} Let $K$ be a $p$-adic field (\textit{i.e.} a finite extension of $\QQ_p$) and let $R$ be its valuation ring. We define a canonical
commutative dga $R\Gamma_{\dR}(X)$ for the de Rham cohomology of a dagger space
$\mathcal{X}$ over $K$.
	 Consider the
following algebra
	\[
		W_n:=\{\sum_\nu a_\nu T^\nu\in K[[T_1,...,T_n]]|\exists \rho>1, \
|a_\nu|\rho^{|\nu|}\to 0\}\ .
	\]
	According to Grosse-Kl\"onne \cite{Gro:00a} a $K$-algebra $A$ is a \emph{dagger algebra}
if it is a quotient of $W_n$ for some $n$. To such an $A$ we can associate the
spectrum of maximal points $\Spm(A)$ which is a $G$-ringed space. One has a
universal $K$-derivation of $A$ into finite $A$-modules, $d :A\to
\Omega^1_{A/K}$ giving  rise to de Rham complex $\Omega_{\mathcal{X}/K}$ on a
general dagger space $\mathcal{X}$. Assuming $\mathcal{X}$ to be smooth we can
set
	\[
		H^n_\dR(\mathcal{X}):=H^n(\mathcal{X},\Omega_{\mathcal{X}/K})
	\]
	It follows from Proposition~\ref{prop:thomsullivan} and
\S~\ref{num:enough_points} that the complex
$R\Gamma_{\dR}(\mathcal{X}):=\Gamma(\mathcal{X},\tGdm\Omega_{\mathcal{X}/K}^\bullet)$
is a functorial commutative dga.

Now let  $X$ be a smooth $R$-scheme. We can associate to it two different dagger
spaces: one is the dagger analytification $(X_K)^\dag$ of its generic fiber; the
other is the Raynaud fiber $(X^w)_K$ of the weakly formal scheme $X^w$
associated to $X$. There is a natural inclusion $(X^w)_K\subset( X_K)^\dag$.
Further there is a map of sites $\iota :(X_K)^\dag\to X_K$ as in the classical
analytification case.
\end{num}
\subsection{Rigid cohomology}
We recall the construction given by Besser as rephrased in \cite{Tam:Kar11} since
there are some simplification. For the sake of the readers we give all the
needed definitions. We fix a  a $p$-adic field  $K$ and denote  by $R$ (resp. $k$)  is its valuation ring (resp. its residue field). 
%
%
%
\begin{num}
	After the work Grosse-Kl\"onne one can compute the rigid cohomology of Berthelot via dagger spaces \cite{Gro:00a}. The method is as follows.  Let $X$ be  a  smooth $k$-scheme, then we can choose a closed embedding  $X\to
\mathcal{Y}$ in a weak formal  $R$-scheme 
$\mathcal{Y}$ having smooth special fiber ${\mathcal{Y}}_k$.  We
call such an embedding  a \emph{rigid pair}  and we denote it by
$(X,\mathcal{Y})$.  There is a specialization map $sp:\mathcal{Y}_K\to {\mathcal{Y}}$, where $\mathcal{Y}_K$ is the generic fiber of $\mathcal{Y}$. We write $]X[_{\mathcal{Y}}:= sp^{-1}(X)$, called the \emph{tube of $X$ in } $\mathcal{Y}$.

	A morphism of rigid pairs $(X,\mathcal{Y})$, $(X',\mathcal{Y}')$ is a
commutative diagram
	\begin{equation*}
	\xymatrix{
	]X[_{\mathcal{Y}} \ar[d]_{sp}\ar[r]^{F}&   ]X'[_{\mathcal{Y}'}
\ar[d]^{sp}\\
	X\ar[r]_{f} &  X '}
	\end{equation*}
  We denote  by $\rm RP$ the category of rigid pairs.

 The datum of a rigid pair $(X,\mathcal{Y})$ is sufficient to compute  the rigid
cohomology of $X$ (with $K$ coefficients)  as follows
	\[		H_\rig^n (X/K)= H_\dR^n( ]X[_{\mathcal{Y}}
)=H^n(]X[_{\mathcal{Y}},\Omega^\bullet_{ ]X[_{\mathcal{Y}} /K} )\ .
	\]

The de Rham complex $\Omega^\bullet_{ ]X[_{\mathcal{Y}} /K}$ is
functorial in $(X,\mathcal{Y})$ and its cohomology is independent up to
isomorphism of the choice of $\mathcal{Y}$. Since the tube of $X$ in $\mathcal
Y$ is a smooth dagger space we get $H_\rig^n
(X/K)=H^n(R\Gamma_{\dR}(\tube{X}{\mathcal Y}))$ (see \ref{num:dagger}). 
\end{num}
\begin{prop}\label{prop:rig}
	\begin{enumerate}
		\item For any $p$-adic field $K$ with residue field $k$ there exists   a ring object $R\Gamma_{\rig,K}$ in the
 category $\DMte(\Spec k,\QQ)$ that represents rigid cohomology (with coefficients in $K$): \textit{i.e.} for any
affine and smooth $k$-scheme $X$,  there is a canonical rational commutative dga
$R\Gamma_{\rig,K}(X)$ 
		such that $H^i(R\Gamma_{\rig,K}(X))\cong
H^i_\rig(X/K)$. (The same holds if we replace the coefficient ring $\QQ$ by any field $L$ s.t. $\QQ\subset L\subset K$) 
\item Let $X$ as above and $(X,\mathcal Y)$ be a rigid pair. Then there is
 a commutative dga $\widetilde{R\Gamma}_\rig(X,\mathcal Y)$  together
with a diagram of dga quasi-isomorphisms
		\[
			R\Gamma_{\rig,K}(X)\gets {R\Gamma}_\rig(X,\mathcal{Y})\to
R\Gamma_{\dR}(\tube{X}{\mathcal{Y}})
		\]
		 functorial in the pair $(X,\mathcal{Y})$.
		\item (Base change) Let $\rho:R\to R'$ be a finite map of complete
discrete valuation rings. Let $k$ (resp. $k'$) be the residue field of $R$ (resp. $R'$). Let  $X$ be
a $k$-scheme then there is a canonical (both in $X$ and $R$) quasi-isomorphism 
		\[
			K'\otimes_K R\Gamma_{\rig,K}(X)\rightarrow
R\Gamma_{\rig,K'}(X_{k'})\ .
		\]
		The latter induces an isomorphism in $\DMte(\Spec k,\QQ)$
$$
R\Gamma_{\rig,K} \otimes_K K' \rightarrow f_*(R\Gamma_{\rig,K'})
$$
where $f:\Spec k'\to \Spec k$ is the map induced by $\rho$
 and $R\Gamma_{\rig,?}$ denotes the object of point (1).
		\item There exists a canonical $\sigma$-linear endomorphism of
$R\Gamma_{\rig,K_0}(X)$  inducing the Frobenius on cohomology: it is defined as the composition of
\begin{equation}\label{eq:frobeniusdefinition}
	R\Gamma_{\rig,K_0}(X) \xrightarrow{\id\otimes 1} R\Gamma_{\rig,K_0}(X) \otimes_{\sigma} K_0
	 \xrightarrow{b.c.} R\Gamma_{\rig,K_0}(F^*X)\xrightarrow{rel. Frob.} R\Gamma_{\rig,K_0}(X)
\end{equation}
where $b.c.$ stands for the base change morphism of point (3); $F$ is the Frobenius of $\Spec k$; $F^*X$ is the base change of $X$ via $F$; the last map on the right is the relative Frobenius.
	\end{enumerate}
\end{prop}
\begin{proof}
	The details are given in \cite[4.9, 4.21, 4.22]{Bes:Syn00}. Since we
adopt the language of dagger spaces there are some formal differences. For the
sake of the readers we give the necessary modifications.	
	To obtain a complex functorial in $X$ we have to take a colimit on some
filtered category. The category of pairs $(X,\mathcal Y)$ with
$X$ fixed is not filtered. Hence we have to introduce the following categories. 
	We define the set $RP_X$ (resp. $RP_{(X,\mathcal Y)}$) of diagrams
$X\xrightarrow{f} X'\to {\mathcal Y}'$ (resp. $(f,F):(X,\mathcal Y)\to
(X',{\mathcal Y}')$ morphism of rigid pairs) where $ (X',{\mathcal Y}')$ is a
rigid pair. Let  $RP_X^0$ (resp. $RP_{(X,\mathcal Y)}^0$) be the subset of
$RP_X$ (resp. $RP_{(X,\mathcal Y)}$) with $f=\id_X$ (resp.
$(f,F)=(\id,\id)$) 
	
	Now we can form the category $SET^0_X$ (resp. $SET^0_{(X,\mathcal{Y})}$)
with objects the finite subsets of $RP_X$ (resp. $RP_{(X,\mathcal Y)}$) having
non-empty intersection with $RP_X^0$ (resp. $RP_{(X,\mathcal Y)}^0$); morphisms
are inclusions. 	 For instance an element of $SET^0_X$ is a finite family
of diagrams  $X\xrightarrow{f_a} X_a'\to {\mathcal Y}_a'$, $a\in A$ (finite
set), such that $f_{a_0}=\id$ for some $a_0 \in A$. To such an object we can
associate the complex $R\Gamma_{\dR}(\tube{X}{{\mathcal Y}_A'})$, where ${\mathcal
Y}_A'=\prod_a {\mathcal Y}_a'$. The categories ${\rm SET}^0_X$ and ${\rm
SET}^0_{(X,\bar X,\sf P)}$ are filtered.
	
	Having this said we define
	\[
		R\Gamma_{\rig,K}(X):=\colim_{A\in {\rm
SET}^0_X}R\Gamma_{\dR}(\tube{X}{{\mathcal Y}_A'})\qquad R\Gamma_{\rig}(X,\mathcal
Y):=\colim_{A\in {\rm SET}^0_{(X,{\mathcal Y})}}R\Gamma_{\dR}(\tube{X}{{\mathcal
Y}_A'})\ .
	\]
Now one can follow word by word the proof of Besser. 
\end{proof}
%

%
%
\begin{prop}\label{prop:rigidspectrum}
There exists a motivic ring spectrum $\E_{\rig,K}$ whose components are all equal to the complex $R\Gamma_{\rig,K}$ and whose stability class is induced by
$\dlog$ such that 
		\[
			H^n_{\rig}(X/K)\cong
\E_{\rig,K}^{n,i}(X):=\Hom_{\DMt(k,\QQ)}(M(X),\E_{\rig,K}(i)[n])\ .
		\]
\end{prop}
\begin{proof}
 We have to verify the hypothesis of
Proposition~\ref{prop:ring_spectrum_existence} for the family
$E_i:=R\Gamma_{\rig,K}$. First we need to define a morphism of complexes $\QQ[0]\to
R\Gamma_{\rig,K}(\mathbb{G}_{m,k})(1)[1]$. We argue as in the de Rham case.
Let us denote by $X=\mathbb{G}_{m,R}$.
Then the de Rham cohomology of the dagger
space $(X^w)_K$ computes the
($K$-linear) rigid cohomology of  $X_k=\mathbb{G}_{m,k}$ and there is a
canonical map from
$R\Gamma_{\dR}((X^w)_K)$ to $R\Gamma_{\rig,K}(X_k)$. We can apply the construction of
\ref{prop:godement} to the inclusion $\Omega^1_{(X^w)_K/K}[-1]\subset
\Omega^1_{(X^w)_K/K}$ and we obtain (as in example~\ref{ex:dlog}) an element
$\dlog$ of $R\Gamma_{\dR}((X^w)_K)$ of degree 1.
\end{proof}

\begin{rem}\label{rem:bc_spectrum_Erig}
With the notations of point (3) of Proposition \ref{prop:rig},
 one gets a canonical base change isomorphism in $\DMB(k)$:
$$
\E_{\rig,K} \otimes_K K' \xrightarrow{\ \sim\ } f_*(\E_{\rig,K'}).
$$
In the sequel, we will simply denote by $\E_{\rig}$
 (resp. $R\Gamma_{\rig}$)
 the ring spectrum $\E_{\rig,K_0}$
 (resp. the complex $R\Gamma_{\rig,K_0}$).
\end{rem}
\subsection{Absolute rigid cohomology}
\begin{num}
Along the lines of \cite{Beu:Not86} and \cite{Ban:Syn02} we are going to define the analogue of absolute Hodge theory in the setting of rigid cohomology. Let $k$ be a perfect field of characteristic $p$. We denote by $\Fisoc$ the category of $F$-isocrystals (defined over $k$): \textit{i.e.} finite dimensional $K_0$-vector spaces together with a $\sigma$-linear automorphism. This is a tensor category with unit object $\un$ given by $K_0$ together with $\sigma$. For any $I\in \Fisoc$ we denote by $I(n)$ the $F$-isocrystal having the same vector space $I$ and Frobenius multiplied by $p^{-n}$. We would like to define the absolute rigid cohomology of a $k$-scheme $X$ as follows
$$	
H^n_{\phi}(X,i):=\Hom_{D^b(\Fisoc)}(\un,R\Gamma(X)(i)[n])
$$
where $R\Gamma(X)$ is a complex of $F$-isocrystals such that $H^n(R\Gamma(X))=H_\rig^n(X)$ together with its Frobenius endomorphism. Since we do not know how to construct directly $R\Gamma$ we follow the strategy of Beilinson in \textit{loc.cit.} and deduce its existence from proposition~\ref{prop:rig}.

Let $C^b_\rig$ be the category of bounded complexes of $K_0$-vector spaces $M$ together with a quasi-isomorphism $\phi:M^\sigma =M\otimes_{K_0,\sigma}K_0 \to M$. We define homotopies  (resp. quasi-isomorphisms) between objects in $C^b_\rig$ to be morphisms in $C^b_\rig$ such that they are homotopies (resp. quasi-isomorphisms) of the underling complexes of $K_0$-vector spaces. Then we can define the category $K^b_\rig$ to be the category $C^b_\rig$ modulo the null-homotopic morphisms. 
\end{num}
\begin{lm}
\begin{enumerate}
	\item the category $K^b_\rig$ is triangulated;
	\item  the localization $K^b_\rig[\mathcal{A}^{-1}]$ of the category $K^b_\rig$ 
    by the subcategory $\mathcal{A}$ of acyclic objects exists and it is a triangulated category too;
    \item let $D^b_\rig\subset K^b_\rig[\mathcal{A}^{-1}] $ be the full subcategory of complexes whose cohomology objects (w.r.t. the usual $t$-structure on complexes) are in $\Fisoc$. Then there is a natural equivalence of categories $\iota: D^b(\Fisoc)\to D^b_\rig$.
   	
\end{enumerate}
\end{lm}
\begin{proof}
We leave to the reader to check that all the arguments given in \cite[\S\S 1,2]{Ban:Syn02} (or \cite[\S~2]{ChiCicMaz:Cyc10}) can be adapted to our (much simpler) setting. We limit ourselves to make explicit the formulas for the $\Hom$ groups in $D^b(\Fisoc), D^b_\rig$.

Let $M,N$ be two bounded complexes of $F$-isocrystals. Remind that $\Fisoc$ has internal Hom so that we can form the internal Hom complex $\uHom^\bullet(M,N)$ with Frobenius $\phi_{M,N}$. Consider the following morphism of $\QQ_p$-linear\footnote{These are not $K_0$-linear since (in general) the Frobenius is not.} complexes
$$
\xi_{M,N}:\uHom^\bullet(M,N)\to \uHom^\bullet(M,N),\quad x\mapsto x-\phi_{M,N}x\ . 
$$ 
Then we can prove as in \cite[proposition~1.7]{Ban:Syn02} that
\begin{equation}
\Hom_{D^b(\Fisoc)}(M,N[i])\cong H^{i-1}(\Cone \xi_{M,N})\ .
\end{equation}

Similarly given two complexes $M,N$ in $C_\rig^b$ we define the morphism of complexes
$$
\xi_{M,N}':\Hom^\bullet(M,N)\to \Hom^\bullet(M^\sigma,N),\quad x\mapsto x\circ \phi_M-\phi_{N}\circ (x\otimes_\sigma 1)\ . 
$$
\footnote{Mind that we cannot use $\uHom$ because there is no internal Hom in $C^b_\rig$. This is due to the fact that the Frobenius is only a quasi-isomorphism.}Then the Hom groups in $D_\rig^b$ can be computed as follows
\begin{equation}\label{eq:extdbrig}
\Hom_{D^b_\rig}(M,N[i])\cong H^{i-1}(\Cone \xi_{M,N}')\ .
\end{equation}

Now it is easy to check that given two $F$-isocrystals $M,N$ we have
\begin{equation}\label{eq:extfisoc}
	\Ext^i_{\Fisoc}(M,N)\cong \Hom_{D^b_\rig}(\iota M,\iota N[i])
\end{equation}
and the faithfulness of $\iota$ follows.
\end{proof}
\begin{df}
Let $X$ be an algebraic $k$-scheme. We define the absolute rigid cohomology as 
$$	
H^n_{\phi}(X,i):=\Hom_{D^b_\rig}(\un,R\Gamma_\rig(X)(i)[n]) \ .
$$
It follows from the equivalence $\iota$ of the above lemma that the same formula holds in  $D^b(\Fisoc)$ for some object $R\Gamma(X)$ corresponding to $R\Gamma_\rig(X)$. 
\end{df}
\begin{cor} \label{cor:absolute_rig_sequence}
	There is a natural spectral sequence
	\begin{equation}\label{eq:spseq absolute rigid}
	E^{pq}_2=\Ext^p_{\Fisoc}(\un,H^q(X)(i))\ \Rightarrow   H^{p+q}_{\phi}(X,i)\ 
	\end{equation}
  degenerating to the following short exact sequence
 \[
0\to H^1_\rig(X)/\Im (\id-\phi/p^i )  \to H^{n,i}_\phi(X)\to H^n_\rig(X)^{\phi=p^i}\to 0\ .
\] 
\end{cor}
\begin{proof}
The existence of the spectral sequence follows from the formula \eqref{eq:extfisoc}.  By \eqref{eq:extdbrig} it is   concentrated in the columns $p=0,1$ so that it gives
 short exact sequences.
\end{proof}

\begin{prop}\label{prop:absrigidspectrum}
There exists a motivic ring spectrum $\E_{\phi}\in\DMB(k)$ representing the absolute rigid cohomology, \textit{i.e.} 
		\[
			H^n_{\phi}(X,i)\cong
\E_{\phi}^{n,i}(X):=\Hom_{\DMB(k)}(M(X),\E_{\phi}(i)[n])\ .
		\]
	\end{prop}
\begin{proof}
 By point (4) of proposition \ref{prop:rig} we can define a family of morphism of presheaves of complexes
	\[
		R\Gamma_\rig\xrightarrow{\phi/p^i}R\Gamma_\rig \ .
	\]
	We claim that the latter induces
	%
	%
	a morphism of ring spectra
	\[
			\E_\rig\xrightarrow{\Phi}\E_\rig\ .
	\]
	Indeed it is sufficient to notice that $(\phi\otimes 1)\circ \dlog =p \dlog$, where $\dlog:\QQ(\GG)[-1]\to \E_\rig(1)$  is the stability class of the rigid spectrum.
	
	Now we can define $\E_{\phi}$ to be  the homotopy limit 
of the
following diagram of ring spectra
\begin{equation} \label{eq:Ephi}
\xymatrix{
\E_{\rig}\ar@<3pt>^{\Phi}[r]\ar_{\id}[r] & \E_{\rig}.
}
\end{equation}

the limit exists by  \ref{num:homotopy_limits}.

To conclude the proof  note that  $\E_{\phi,i}$ is quasi-isomorphic to the cone $\Cone(\id -\phi/p^i)$ (up to a shift!).  Then  it is sufficient to compare \eqref{eq:extdbrig} and  \eqref{eq:compute_rep}.
\end{proof}

\begin{rem}
According to the preceding proof,
 one gets a canonical distinguished triangle of $\DMB(k)$:
\begin{equation} \label{eq:triangle_Ephi}
\E_\phi \rightarrow \E_{\rig} \xrightarrow{\id-\Phi} \E_\rig
 \xrightarrow{+1}
\end{equation}
which induces the short exact sequences of the preceding Corollary.
In particular, these exact sequences are functorial with respect
 to the motive of $X$.
\end{rem}

\subsection{Syntomic cohomology}
\begin{num} \label{num:syntomicrecall}
	Let $X$ be a smooth $R$-scheme. With the notation of \S~\ref{num:dagger}, there is a map of commutative dga
	\[
		\spe_X:E_\dR(X_K)\rightarrow R\Gamma_{\dR}((X^w)_K)=R\Gamma_\rig(X_k,X^w)
	\]
	inducing the specialization on cohomology and functorial in $X$. Details can be found in
\cite[\S\S~3.3,5.3]{Tam:Kar11}.
	
	Now we can recall the definition of syntomic cohomology $H^n_\syn(X,i)$
of $X$: it  is the cohomology of a complex $R\Gamma_\syn(X,i)$ defined as the
homotopy limit of the following diagram
	\begin{equation*}
	\xymatrix@C=-12pt{
	& R\Gamma_\rig(X_k)& & R\Gamma_{\rig,K}(X_k) & &R\Gamma_{\dR}(\tube{X_k}{X_w})  & &
E_\dR'(X_K)\\
 	R\Gamma_\rig(X_k)\ar[ru]^\id\ar@{=}[rr] & & \ar[lu]^{\phi/p^i} R\Gamma_\rig(X_k)\ar[ru]& &
\ar[lu] R\Gamma_\rig(X_k,X^w)\ar[ru] & & \ar[lu] E_\dR(X_K)\ar[ru]& &\ar[lu]
E_{\FdR,i}(X_K) } \ .
	\end{equation*} 
	(cf. \cite{Bes:Syn00}, \cite{ChiCicMaz:Cyc10}). To be precise Besser uses the cone of $\phi-p^i\id$ instead of $\id -\phi/p^i$.
	\end{num}
%
%
	\begin{prop}\label{prop:syntomicspectrum}
Let $R$ be the valuation ring of a $p$-adic field $K$, then
there exists a ring spectrum
$\E_{\syn}$ in $\DMB(R,\QQ_p)$
 representing  the syntomic cohomology defined by Besser, i.e. 
for any smooth $R$-scheme $X$ and any integer $n$, 
 there is a canonical isomorphism 
\[
	H^n_{\syn}(X,i)\cong
\E_{\syn}^{n,i}(X):=\Hom_{\DMB(R,\QQ_p)}(M(X),\E_{\syn}(i)[n])\ .
\]	

In particular all the results of section 2 apply to syntomic cohomology.
	\end{prop}
\begin{proof}
	 By construction  the absolute rigid spectrum $\E_\phi$ maps to $\E_\rig$ and so to the base change  $\E_{\rig,K}$.  By the six functor formalism we get the following functors
	 \[
	   i_*:\DMB(k,\QQ_p)\to \DMB(R,\QQ_p)\ ,\ j_*:\DMB(K,\QQ_p)\to \DMB(R,\QQ_p) 
	 \]
	 induced by the usual closed (resp. open) immersion of schemes $i:\Spec(k)\to \Spec (R)$ (resp. $j:\Spec(K)\to \Spec(R)$).
Then we define $\E_{\syn}$ as the  homotopy limit (in the category of
ring spectra) of the following diagram
$$
 i_*\E_\phi\to i_*\E_{\rig,K}\gets a\to b\gets c\to d \leftarrow j_*\E_\FdR
$$
where $a,b,c,d$ are the ring spectra induced by $
  E_\rig(X_k,X^w)$, $R\Gamma_{\dR}(\tube{X_k}{X_w})$, $E_\dR(X_K)$,
$E_\dR'(X_K)$, respectively: we leave to the reader the verification that they are ring
spectra following the same proof as the one of \ref{prop:rigidspectrum}.

To conclude the proof it is sufficient to note that an homotopy limit of a diagram of  Morel motives is also a Morel motive.
\end{proof}
\begin{rem}
Given a complete discrete valuation ring	 $R$ with residue field $k$ and fraction field $K$,
 such that $R/W(k)$ is finite, we get a map of ring spectra
 in $\DMB(k)$:
$$
a_0:\E_\phi \rightarrow \E_{\rig,K_0}
 \rightarrow \E_{\rig,K_0} \otimes_{K_0} K
 \xrightarrow{\sim} \E_{\rig,K}
$$
where the last isomorphism comes from Remark \ref{rem:bc_spectrum_Erig}.
Let us put $a=i_*(a_0)$.

Secondly, we get a morphism of ring spectra
 in $\DMB(R)$:
$$
b:j_*\E_{\FdR} \rightarrow j_*\E_{\dR}
 \xrightarrow{\mathrm{sp}} i_*\E_{\rig,K}.
$$
The first map is the canonical morphism,
 and the second one is the specialization map
 induced by the morphism $\spe_X$ of Paragraph \ref{num:syntomicrecall}.

Then the syntomic ring spectrum is characterized
 up to isomorphism by the following homotopy pullback
 square (of morphisms of ring spectra):
\begin{equation}\label{eq:Esyn_pullback}
\xymatrix@=15pt{
\E_{\syn}\ar_\beta[d]\ar^-\alpha[r]
 & j_*\E_{\FdR}\ar^b[d] \\
i_*\E_\phi\ar^-a[r] & i_*\E_{\rig,K}
}
\end{equation}
In other words, one can define $\E_\syn$
 as the homotopy limit of the lower corner of the above diagram 
 -- but this definition
 is less precise than the one given in the proof of the previous proposition  as (in this way)
 $\E_\syn$ is  defined only up to non
 unique isomorphism.

The fact that the preceding square is a homotopy
pullback can be translated into the existence
of a distinguished triangle in $\DMB(R)$:
\begin{equation} \label{eq:syntomictriangle}
\E_\syn
 \xrightarrow{\alpha+\beta} i_*\E_\phi \oplus j_*\E_{\FdR}
 \xrightarrow{a-b} i_*\E_{\rig,K} \xrightarrow{+1} 
\end{equation}
which corresponds to the long exact sequence, for $X/R$ smooth:
\begin{equation} \label{eq:syntomicsequence}
\hdots \rightarrow H_\syn^{n}(X,i)
 \xrightarrow{\alpha_*+\beta_*} H_\phi^{n}(X_k,i) \oplus F^iH^n_\dR(X_K)
 \xrightarrow{a_*-b_*} H^n_{\rig}(X_k/K) 
 \rightarrow  \hdots
\end{equation}
Here, $\alpha_*$ (resp. $\beta_*$) is the usual projection
 map from syntomic cohomology to $\E_\phi^{n,i}(X_k)=H_\phi^n(X_k,i)$
(resp. $F^iH^n_\dR(X_K)$) while $a_*$ is the canonical map
 and $b_*$ is induced by the specialization map from de Rham cohomology
 to rigid cohomology.

Note also that $\E_\syn$ is the homotopy limit of the diagram
 of ring spectra
$$
\xymatrix@=20pt{
& & j_*\E_{\FdR}\ar^b[d] \\
i_*\E_{\rig}\ar@<2pt>^{\Phi}[r]\ar@<-1pt>_{\id}[r]
 & i_*\E_{\rig}\ar[r]
 & i_*\E_{\rig,K}
}
$$
so that we also obtain the following distinguished triangle:
$$
\E_\syn
 \rightarrow i_*\E_{\rig} \oplus j_*\E_{\FdR}
 \rightarrow i_*\E_{\rig} \oplus i_*\E_{\rig,K}
 \xrightarrow{+1} 
$$
which precisely induces the long exact sequence
 originally considered by Besser.
\end{rem}
\begin{rem}
Syntomic cohomology can be functionally extended to diagrams
 of $S$-schemes, as well as rigid cohomology, absolute rigid cohomology
 and filtered de Rham cohomology.
 One should be careful however that the syntomic long exact sequence 
 \eqref{eq:syntomictriangle}
 can be extended only to the case of diagrams of smooth $S$-schemes.
\end{rem}
\subsection{Localizing syntomic cohomology}
\begin{num} \label{num:recall_loc}
As the fibred triangulated category $\DMB$ satisfies the
``gluing formalism" (this is called the localization property in
\cite{CD3}, cf. sec. 2.3), we get a canonical distinguished triangle:
\begin{equation}\label{eq:local_Esyn}
i_*i^!(\E_\syn) \xrightarrow{ad'_i} \E_\syn
 \xrightarrow{ad_j} j_*j^*(\E_\syn) \xrightarrow{\partial_i}
 i_*i^!(\E_\syn)[1]
\end{equation}
for $i:\Spec k \rightarrow \Spec R$ and $j:\Spec K \rightarrow \Spec R$
 the natural immersions.
 The maps $ad'_i$ and $ad_j$ are the obvious adjunction maps
 and the map $\partial_i$ is the unique morphism which fits
 in this distinguished triangle (see \cite[2.3.3]{CD3}).
\end{num}

\begin{rem}
One can be more precise about the gluing formalism:
 given any object $M$ of $\DMB(R)$, there exists a unique
 distinguished triangle of the form
$$
M_k \rightarrow M \rightarrow M_K \xrightarrow{\partial} M_k[1]
$$
such that $M_k$ (resp. $M_K$)
 has support in $\Spec k$, \textit{i.e.} $j^*M_k=0$
  (resp. in $\Spec K$, \textit{i.e.} $i^!(M_K)=0$).
This means that there exists a canonical isomorphism
 of that triangle with the following one:
$$
i_*i^!(M) \xrightarrow{ad'_i} M \xrightarrow{ad_j} j_*j^*(M)
 \xrightarrow{\partial_i} i_*i^!(M)[1].
$$ 
\end{rem}
\begin{num}
    Let us introduce yet another spectrum: we consider the map
    $$
    a_0:\E_\phi \rightarrow E_{\rig,K}
    $$
    which is defined at the level of the underlying model category,
     and take its homotopy fiber $\check \E_\phi$.
		In particular, we have a canonical morphism:
		$i_*E_{\rig,K} \xrightarrow{\partial_{a}}  i_*\check\E_\phi[1]$.
\end{num}
\begin{prop}\label{prop:syntomic localization}
Consider the above notations.
Then the syntomic spectrum is equivalent to the homotopy fiber of the morphism
\[
\check \spe:j_*\E_\FdR \xrightarrow{b} i_*\E_{\rig,K}
 \xrightarrow{\partial_{a}}    i_*{\check \E_\phi}[1]\ .
	\]
Moreover, there are canonical identifications
		\[
		i^!\E_{\syn}= \check\E_{\phi},
		\qquad
j^*\E_{\syn}=\E_\FdR
		\]
through which the localization triangle \eqref{eq:local_Esyn}
 is identified with
$$
i_*{\check \E_\phi} \rightarrow \E_\syn
 \rightarrow j_*\E_\dR \xrightarrow{\check \spe} i_*{\check \E_\phi}[1].
$$
\end{prop}

\begin{rem}
In fancy terms, the generic fiber of $\E_\syn$ is the
 ring spectrum $\E_\FdR$.
 While we cannot compute the special fiber of $\E_\syn$,
 its exceptional special fiber is the ring spectrum
 which is "the image of absolute rigid cohomology
 in rigid cohomology" and $\E_\syn$ is obtained by gluing these
 two ring spectra.
\end{rem}
\begin{proof}
By definition of $\check \E_\phi$, there is a canonical distinguished triangle in $\DMB(k)$:
$$
\check \E_\phi \xrightarrow{\nu_0}
 \E_\phi \xrightarrow{a_0} E_{\rig,K}
 \xrightarrow{\partial_{a_0}} \check \E_\phi[1]
$$
which induces the following triangle after applying $i_*$:
$$
i_*\check \E_\phi \xrightarrow{\nu}
 i_*\E_\phi \xrightarrow{a} i_*E_{\rig,K}
 \xrightarrow{\partial_{a}} i_*\check \E_\phi[1].
$$
Now according to the fact the square \eqref{eq:Esyn_pullback}
 is a homotopy pullback, one gets a canonical commutative
 diagram in $\DMB(R)$:
$$
\xymatrix@=15pt{
C(\alpha)\ar[r]\ar_\sim[d]
 & \E_{\syn}\ar_\beta[d]\ar^-\alpha[r]
 & j_*\E_{\FdR}\ar^b[d]\ar[r]
 & C(\alpha)[1]\ar^\sim[d] \\
i_*\check \E_\phi\ar^\nu[r]
& i_*\E_\phi\ar^-a[r] & i_*E_{\rig,K}\ar^{\partial_a}[r]
& i_*\check \E_\phi[1].
}
$$
In other words, we get a distinguished triangle of the form:
$$
i_*\check \E_\phi \rightarrow
 \E_{\syn} \xrightarrow{\alpha} j_*\E_{\FdR}
 \xrightarrow{\check \spe} i_*\check \E_\phi[1].
$$
Finally, according to the above remark,
 one gets a canonical isomorphism of triangles:
$$
\xymatrix@=20pt{
i_*\check \E_\phi\ar[r]\ar_\sim[d]
 & \E_{\syn}\ar_\sim[d]\ar^-\alpha[r]
 & j_*\E_{\FdR}\ar^\sim[d]\ar^{\check \spe}[r]
 & i_*\check \E_\phi[1]\ar^\sim[d] \\
i_*i^! \E_{\syn}\ar^{ad_i}[r]
 & \E_{\syn}\ar^{ad_j}[r]
 & j_*j^*(\E_{\syn})\ar^{\partial_i}[r]
& i_*i^! \E_{\syn}[1]\ .
}
$$
\end{proof}
\begin{rem}[The work of Tamme]
 The  {\em relative cohomology theory} $H^*_\rel(X,*)$ of \cite{Tam:Kar11} is
represented by the (generalized) cone of the diagram
$$
i_*\E_{\rig,K}\gets a\to b\gets c\to d \leftarrow j_*\E_\FdR
$$
where we use the notation of the proof of Proposition~\ref{prop:syntomicspectrum}. This is roughly a Cone of a
morphism of ring spectra $A\to B$, hence it is not a ring spectrum and
in particular there is no unit section.

It follows by the localization sequence  that this cohomology theory is
represented by the cone of the canonical adjunction map $\E_\syn\to
i_*i^!\E_\syn=i_*\E_\phi$.
\end{rem}
\begin{ex}\label{ex:computation1}
Let $S=\Spec\big(W(k)\big)$ (for simplicity) and  $X$ be the connected component of the N\'eron model of an elliptic curve with multiplicative reduction, i.e. $X$ is an $S$-group scheme such that its generic fiber is an elliptic curve and the special fiber is isomorphic to $\mathbb{G}_m$. Then $X/S$ is smooth and we can easily compute the long exact sequence for syntomic cohomology. For instance we get
$$
0\to H_\rig^0(X_s)\xrightarrow{a} H_\rig^0(X_s)\oplus H_\rig^0(X_s)\to H_\syn^{1,1}(X)\to
$$

$$
\to H_\rig^1(X_s)\oplus F^1H_\dR^1(X_\eta)\xrightarrow{b}  H_\rig^1(X_s)\oplus H_\rig^1(X_s)\to H_\syn^{2,1}(X)\to F^1H_\dR^2(X_\eta)\to 0
$$
where $a(x)=(x-\phi(x)/p,-x)$ is injective and $b(x,y)=(0,y-x)$. It follows that $H_\syn^{n,1}\cong K^2$ (as $\QQ_p$-vector spaces) for $n=1,2$.

The same result can be obtained using the localization triangle: explicitly we get the following exact sequence
$$	
0\to H_{\syn,s}^{1,1}(X)\to H_\syn^{1,1}(X)\to F^1H^1_\dR(X_\eta)\xrightarrow{\delta}H_{\syn,s}^{2,1}(X)\to H_\syn^{2,1}(X)\to F^1H^2_\dR(X_\eta)\to 0\ .
$$
Here $H_{\syn,s}^{1,1}(X)$ stands for $\Hom(\QQ(X),i_*i^!\E_\syn(1)[1])$. Using  proposition~\ref{prop:syntomic localization} we get $H_{\syn,s}^{n,1}(X)=H^{n-1}_\rig(X_s)$ for $n=1,2$. We also get that $\delta$ is the zero map. For a complete account on the de Rham/rigid cohomology of abelian varieties and their reduction we refer to \cite{Le-:Coh86}.
\end{ex}	
\begin{ex}[Semistable elliptic curve] \label{ex:computation_semistable_curve}
Let $X/S$ be an elliptic curve such that
$X_k$ is a nodal cubic. We assume that the
singular point $x_0\in X_k$ is $k$-rational. The above remark give a recipe to
compute (or approximate) the syntomic cohomology of $X$
$$
\E_\syn^{n,i}(X):=\Hom_{\DMt(S,\QQ_p)}(M(X),\E_{\syn}(i)[n])
$$
 where  $M(X)=f_!f^!(\QQ_S)$ and $f:X\to S$ is the structural morphism.
 Let us compute $i^*(M(X))$. Given the pullback
 square:
 $$
 \xymatrix{
 X_k\ar^l[r]\ar_{f_0}[d] & X\ar^f[d] \\
 \Spec k\ar^i[r] & S
 }
 $$
 one has a canonical exchange map:
 $$
 i^*f_!f^!(\QQ_S) \simeq f_{0!}l^*f^!(\QQ_S)
  \rightarrow f_{0!}f_0^!i^*(\QQ_S)=f_{0!}f_0^!(\QQ_k)=M(X_k)
 $$
 (the first iso is due to the base change theorem of the six functors
  formalism). This map is an isomorphism in the two following cases:
 \begin{itemize}
 \item $f$ is smooth,
 \item $X$ is regular and $f$ is quasi-projective (and so in our case).
  
  In the second case, this is due to the absolute purity theorem: as $f$ is quasi-projective,
  it can be factored $f=pi$ where $p:P \rightarrow S$
  is smooth and $i$ is a closed immersion and then one computes:
 $$
 f^!(\QQ_S)=i^!p^!(\QQ_S) \simeq i^!(\QQ_P)(d)[2d] \simeq \QQ_X(d-n)[2(d-n)]
 $$
 the first iso follows as $p$ is smooth and the second one because $i$
 is a closed immersion between regular schemes. 
 Here $d$ (resp. $n$) is the relative dimension of $P/S$ (resp. codimension
 of $i$), so that $d-n$ is the relative dimension of $X/S$.
 \end{itemize}

 Hence for $X$ semistable we have long exact
sequences
$$
i^!\E_\syn^{n,i}(X_k)\to \E_\syn^{n,i}(X)\to
j_*j^*\E_\syn^{n,i}(X)=F^iH_\dR(X_\eta)\to +
$$

The term $i^!\E_\syn^{n,i}(X_k)$ depends only on the special fiber. In this
case it is easy to construct a proper and smooth hypercover $Y_*$ of $X_k$:
let    $\pi:\widetilde{X}_k\to X_k$ be the normalization map, then we may take
$Y_0=x_0\sqcup \widetilde{X}_k$, $Y_1=\pi^{-1}(x_0)$ and $Y_i=\varnothing$ for 
$i>1$. Since $\widetilde{X}_k$ is isomorphic to the projective line we get that
$M(X_k)=\QQ \oplus \QQ[1] \oplus \QQ(1)[2]$ in $\DMB(k,\QQ)$. This decomposition
allows to estimate $i^!\E_\syn^{n,i}(X_k)$. For instance we can compute
$$
 i^!\E_\syn^{n,i}(X_k)=H^{n-1}_\rig (X_k)_K\simeq K \qquad \text{for}\ n=1,2\ .
$$
\end{ex}


\subsection{Syntomic regulator}

\begin{num} \label{rem:syntomic_cycle_class}
By using the general definition of \S~\ref{num:cycle_classes}     we get the syntomic (resp. rigid, de Rham, etc.) cycle classes.
Since all the maps of the homotopy pullback square
 \eqref{eq:Esyn_pullback} are morphisms of monoids in $\DMB(R)$, we get the following commutative diagram:
$$
\xymatrix{
H^n_\syn(X,m)\ar^{\alpha_*}[rr]\ar_{\beta_*}[dd]
 && F^mH^n_\dR(X_K)\ar^{\spe}[dd] \\
& \HB^{n,m}(X)\ar_{\sigma_\phi i^*}[ld]\ar@{}|{(a)}[d]\ar@{}|{(b)}[r]
 \ar|{\sigma_\syn}[lu]\ar^{\sigma_\FdR j^*}[ru]\ar|{\sigma_\rig i^*}[rd] & \\
H^n_{\phi}(X_k,m)\ar[rr] && H^n_{\rig}(X_k/K)
}
$$
where $\sigma_?$ stands for the higher cycle classes relevant to the 
corresponding cohomology , and $i^*$ (resp. $j^*$) denotes the pullback
in motivic cohomology by $i$ (resp. $j$).

\begin{enumerate}
\item 
The part (a) of the above commutative diagram
 simply express the fact that for any smooth $k$-scheme $X_0$,
 the higher cycle class map
$$
\sigma_\rig:\HB^{n,m}(X_0) \rightarrow H^n_\rig(X/K_0)
$$
lands into the part $\phi=p^m$ of rigid cohomology and that
 it admits a canonical lifting to the absolute rigid cohomology
 $H_\phi^{n,m}(X)$ through the canonical surjection
$$
H_\phi^{n,m}(X) \rightarrow H^n_\rig(X/K_0)^{\phi=p^m}
$$
of Corollary \ref{cor:absolute_rig_sequence}.
\item One can deduce from the commutativity of the part (b)
 of the above diagram another proof of the fact,
 already obtained in \cite{ChiCicMaz:Cyc10}, that the specialization map
 $\spe$ is compatible with the specialization map $\spe_{\CH}$
 in Chow theory as defined in \cite[\textsection 20.3]{Ful}.
Indeed, in the case $n=2m$, (b) can be rewritten as follows:
$$
\xymatrix@R=10pt@C=36pt{
& \CH^m(X_K)\ar^-{\sigma_\FdR}[r]\ar@{-->}^{\spe_{\CH}}[dd]
 & F^{2m}H^m_\dR(X_K)\ar^{\spe}[dd] \\
\CH^m(X)\ar_{i^*}[rd]\ar^{j^*}[ru] && \\
& \CH^m(X_k)\ar^-{\sigma_\rig}[r] & H^m_{\rig}(X_k/K)
}
$$
and the assertion follows as $j^*$ is surjective and $\spe_{\CH}$
 is the unique morphism making the left hand side commutative.
\item (Concerning the terminology)
The term ``higher cycle classes'' comes from the theory of higher Chow groups
 -- which, for smooth $R$-schemes, coincide rationally with
 Beilinson motivic cohomology according to \cite[14.7]{Lev}.

The term ``syntomic regulator'' has been introduced by M.~Gros
 in \cite{Gros}. It comes from the intuition that
 syntomic cohomology is an analogue of Deligne cohomology
 and that one can transport the setting of Beilinson's conjectures
 from Deligne cohomology to syntomic cohomology.
 One should be careful however that in the case of Deligne cohomology
 and if $(2m-n)=1$, then the higher cycle class map is only a part
 of the regulator (see \cite[\textsection 3.3]{Sou}).
\end{enumerate}
\end{num}

\begin{rem}
The syntomic Chern classes are constructed as in \S~\ref{num:chern_classes}. These are determined by 
 the first Chern class $c_1$
 of the canonical line bundle of $\PP^1_R$.
 According to our construction of the syntomic ring spectrum,
 this is nothing else than the class $\dlog$. 
One deduces that the Chern classes obtained here in syntomic cohomology
 coincide with the one previously constructed
 by A.~Besser in \cite{Bes:Syn00}.
\end{rem}
\begin{prop}
Let $f:Y \rightarrow X$ be a projective morphism between
 smooth $R$-schemes, and denote by $f_k$ (resp. $f_K$)
 its special (resp. generic) fiber.
Then the following diagram is commutative:
$$
\xymatrix@C=8pt@R=20pt{
H_\syn^{n}(Y,i)\ar^-{\alpha_*+\beta_*}[r]\ar_{f_*}[d]
 & H_\phi^{n}(Y_k,i) \oplus F^iH^n_\dR(Y_K)\ar^-{a_*-b_*}[r]
     \ar^{f_{k*}+f_{K*}}[d]
 & H^n_{\rig}(Y_k/K)\ar[r]\ar^{f_{k*}}[d]
 & H_\syn^{n+1}(Y,i)\ar^{f_*}[d] \\
H_\syn^{n-2d}(X,i-d)\ar_-{\alpha_*+\beta_*}[r]
 & H_\phi^{n-2d}(X_k,i-d) \oplus F^{i-d}H^{n-2d}_\dR(X_K)\ar_-{a_*-b_*}[r]
 & H^{n-2d}_{\rig}(X_k/K)\ar[r] & H_\syn^{n-2d+1}(X,i-d)
}
$$
where the lines are given by the exact sequences
  \eqref{eq:syntomicsequence}.
\end{prop}
\begin{proof}
    Applying the same formalism to the motivic ring spectra
     $\E_\FdR$, $\E_{\rig,K}$, $\E_\phi$, one obtains Gysin morphisms
     on their cohomology, satisfying the preceding properties.
    Moreover, using the distinguished triangle \eqref{eq:syntomictriangle}
     of $\DMB(R)$, one gets the  result.
\end{proof}

\begin{num}
    Recall that in \S~\ref{num:4theories} we have associated four theories (cohomology, homology, coho. with compact support, BM homology) to any motivic ring spectrum.
\begin{enumerate}
\item We get syntomic theories and the higher cycle class \eqref{eq:syntomic_higher_cycle} also  for singular $R$-schemes\footnote{Recall that for singular schemes,
 Beilinson motivic cohomology is defined after \cite{CD3} and \cite{Cis}
 as the graded part of homotopy invariant K-theory for the
 $\gamma$-filtration.}.

When focusing attention to Chow theory, one gets in particular:
\begin{itemize}
\item $X$ regular: $\sigma_\syn:\CH^n(X) \rightarrow H^{2n}_\syn(X,n)$.
\item $X$ regular quasi-projective: 
$\sigma_\syn:\CH_n(X) \rightarrow H_{2n}^{\syn,\BM}(X,n)$.
\end{itemize}
The second point follows from the fact
 $\HB^{n,i}(X) \simeq H_{2d-n,d-i}^{\mathcyr B,\BM}(X)$
 where $d$ is the (Krull) dimension of $X$ according to the motivic
 absolute purity theorem (\cite[14.4.1]{CD3}).
\item When the base scheme is $S=\Spec k$, we get  rigid (resp. absolute rigid) theories
 associated with $\E_{\syn,K}$ (resp. $\E_\phi$) and regulators for these
 theories.

In the case $K=K_0$, the Frobenius operator $\Phi$ of $\E_\rig$
 induces an action of Frobenius on all four theories, compatible with
 the regulator.
Moreover,
 the distinguished triangle \eqref{eq:triangle_Ephi} yields long exact
 sequences in all four theories.
\item When $S=\Spec K$, we get the  de Rham theory (resp. filtered de Rham)  associated with   $\E_\dR$ (resp. $\E_\FdR$)  equipped with regulators.
 The canonical map $\E_\FdR \rightarrow \E_\dR$ induces natural
 maps of these theories, compatible with regulators.

Consider the specialization map:
$$
\spe:j_*\E_\FdR \rightarrow i_*\E_{\rig,K}.
$$
Given any $R$-scheme $X$ with structural morphism $f$,
 and applying $f_*f^!$ to this map one obtains:
$$
\spe_*:H_n^{\FdR,\BM}(X_K,i) \rightarrow H_n^{\rig,K,\BM}(X_k,i)
$$
using the exchange isomorphisms: $f^!i_*=i'_*f_k^!$
 and $f^!j_*=j'_*f_K^!$.
Similarly, if we apply $f_*f^!$ to the distinguished triangle
 \eqref{eq:syntomictriangle}, one gets the following long exact
 sequence:
\begin{equation} \label{eq:syntomic_BM_sequence}
\hdots \rightarrow H^{\syn,\BM}_{n}(X,i)
 \xrightarrow{\alpha_*+\beta_*} H^{\phi,\BM}_n(X_k,i) \oplus H^{\FdR,\BM}_n(X_K,i)
 \xrightarrow{a_*-\spe_*} H_n^{\rig,K,\BM}(X_k,i)
 \rightarrow  \hdots
\end{equation}
\end{enumerate}
\end{num}

\begin{num}
All the theories considered in the previous paragraph
 satisfy the functorialities described in \S~\ref{num:functo}.
Moreover, regulators are compatible with these functorialities.
 Similarly, the maps $\spe_*$, $\alpha_*$, $\beta_*$, $a_*$
 and $b_*$ considered in part (3) of this example
 are natural with respect to proper covariant and smooth 
 contravariant functorialities.

Moreover, taking care of the functoriality explained in the previous
 remark for motivic BM-homology, one can check the following
 diagram is commutative:
$$
\xymatrix@R=10pt@C=36pt{
& H^\mathcyr B_{n,i}(X_K/K)\ar^-{\sigma_\FdR}[r]
& H_n^{\FdR,\BM}(X_K,i)\ar^{\spe_*}[dd] \\
H^\mathcyr B_{n,i}(X/R)\ar_{i^*}[rd]\ar^{j^*}[ru] & & \\
& H^\mathcyr B_{n,i}(X_k/k)\ar^-{\sigma_\rig}[r] 
& H_n^{\rig,K,\BM}(X_k,i).
}
$$
When $X/R$ is quasi-projective regular with good reduction and $i=2n$,
 one obtains in particular a generalization of the second part
 of Remark \ref{rem:syntomic_cycle_class} (applying the
 motivic absolute purity theorem \cite[14.4.1]{CD3},
 all the motivic BM-homology in the above diagram
 can be identified with Chow groups in that case).

This fact can be extended to the exact sequence
 \eqref{eq:syntomic_BM_sequence} and to its compatibility
 with the regulator in syntomic BM-homology.
\end{num}
\subsection{Rigid syntomic modules} \label{sec:modules}

\begin{num}
The aim of this last section is to apply the theory
 developed in \cite[sec. 7.2]{CD3} to the syntomic
 ring spectrum $\E_\syn$.

Put $S=\Spec R$.
Recall that by construction,
 $\E_\syn$ can be seen as an object of $\Spr(S,\QQ)$
  (Paragraph \ref{num:ring_spectra}).

Let $f:X \rightarrow S$ be any morphism of schemes.
The pullback functor $f^*$ on the category of Tate spectra
 is monoidal. Thus, it obviously induces a functor:
$$
f^*:\Spr(S,\QQ) \rightarrow \Spr(X,\QQ).
$$
In particular, we can define the rigid syntomic ring spectrum
 over $X$ as follows:
$$
\E_{\syn,X}:=f^*(\E_\syn).
$$
The collection of these ring spectra
 defines a cartesian section of the fibered category
 $\Spr(-,\QQ)$ over the category of $R$-schemes.
In particular,
 one can apply \cite[Prop. 7.2.11]{CD3} to it.
In particular, the category of modules 
 over $\E_{\syn,X}$ in $\Sp(X,\QQ)$
 admits a model structure.
\end{num}
\begin{df}
Consider the above notations.

We define the category $\smod_X$
 of \emph{rigid syntomic modules over $X$}
 as the homotopy category
 of the model category of modules over 
 the ring spectrum $\E_{\syn,X}$.
\end{df}

\begin{num}
According to \cite{CD3}, Prop. 7.2.13 and 7.2.18,
 rigid syntomic modules inherit the good functoriality 
 properties of the stable homotopy category
 (in the terminology of \cite[Def. 2.4.45]{CD3}, 
   the category $\smod$, fibered over the category of
	 $R$-schemes, is motivic).
Let us recall briefly the six functors formalisms:
 given a morphism $f:T \rightarrow S$ of $R$-schemes,
 one has two pairs of adjoint functors:
\begin{align*}
& f^*:\smod_S \leftrightarrows \smod_T:f_*\ , \\
& f_!:\smod_T \leftrightarrows \smod_S:f_!\ ,
 \text{ for $f$ separated of finite type,}
\end{align*}
and $\smod_X$ is triangulated closed monoidal.
 We denote by $\otimes$ (resp. $\uHom$) the tensor product
 (resp. internal Hom).
\begin{itemize}
\item $f_*=f_!$ for $f$ proper,
\item \underline{Relative purity}:
 $f^!=f^*(d)[2d]$ for $f$ smooth of constant relative dimension $d$,
\item \underline{Base change formulas}: $f^*g_!=g'_!f^{\prime*}$,
 for $f$ any morphism (resp. $g$ any separated morphism of finite type),
 $f'$ (resp. $g'$) the base change of $f$ along $g$ (resp. $g$ along $f$).
\item \underline{Projection formulas}: $f^!(M \otimes f^*(N))=f_!(M) \otimes N$.
\item \underline{Localization property}: 
 given any closed immersion $i:Z \rightarrow S$
 of $R$-schemes, with complementary open immersion $j$,
 there exists a distinguished triangle of natural transformations as follows:
\begin{align*}
j_!j^! &\rightarrow 1 \rightarrow i_*i^*
 \xrightarrow{\ \partial_i\ } j_!j^![1]
\end{align*}
where the first (resp. second) map denotes the counit (resp. unit) 
 of the relevant adjunction (as in Paragraph \ref{num:recall_loc}).
\end{itemize}
\end{num}

\begin{rem}
An important set of properties is missing in the theory
 of rigid syntomic modules. 

One will say that a syntomic module over $X$ is \emph{constructible}
 if and only if it is compact in the triangulated
 category $\smod_X$.
The category of constructible modules should enjoy
 the following properties: 
\begin{enumerate}
\item they are stable by the six operations
 (when restricted to excellent $R$-schemes),
\item they satisfy Grothendieck duality
 (existence of a dualizing module).
\end{enumerate}
To get these properties, one has only to prove
 the absolute purity for syntomic modules: 
 given any closed immersion
 $i:Z \rightarrow X$ of regular $R$-schemes,
 of pure codimension $c$,
 there exists an isomorphism:
$$
i^!(\un_X)=\un_Z(c)[2c].
$$
\end{rem}

\begin{num} \textit{Syntomic triangulated realization}.--
Applying again \cite{CD3}, Prop. 7.2.13,
 one gets for any $R$-scheme $X$ an adjunction
 of triangulated categories:
$$
L^\syn_X:\DMB(X) \leftrightarrows \smod_X:\mathcal O^\syn_X
$$
such that:
\begin{enumerate}
\item $\mathcal O^\syn_X$ is conservative.
\item For any Beilinson motive $M$ over $X$,
 one has an isomorphism
$$ \mathcal O^\syn L^\syn (M) \simeq M \otimes \E_\syn $$
functorial in $M$.
\item The functor $L^\syn_X$ commutes with the operations 
 $f^*$, $f_!$, $\otimes$.
\end{enumerate}
Let us denote by $\un_X$ the unit object of $\smod_X$.
According to point (2), one obtains a canonical isomorphism:
$$
\Hom_{\smod_X}\!\big(\un_X,\un_X(i)[n]\big)
 \simeq \E_\syn^{n,i}(X)
$$
which is functorial in $X$ and compatible with products.
\end{num}

\begin{rem}
In the preceding section, one has derived
 Bloch-Ogus axioms,
 for syntomic cohomology and syntomic BM-homology,
 from the functoriality of $\DMB$.
 In fact, as in \cite[Ex. 2.1]{BO},
  one can also obtain these axioms from the properties 
	of syntomic modules stated above.
\end{rem}

\begin{num} \textit{Descent properties}.--
According to \cite[sec. 3.1]{CD3},
 the $2$-functor $X \mapsto \smod_X$ can
 be extended to diagrams of $R$-schemes
 (as well as the syntomic triangulated realization).
 Moreover, the pair of functors $(f^*,f_*)$
 can be defined when $f$ is a morphism of diagrams of $R$-schemes.

From \cite[7.2.18]{CD3}, the motivic category
 $\smod$ is separated. Therefore,
 according to \cite[3.3.37]{CD3}, it satisfies h-descent
  (see Paragraph \ref{num:hdescent} for the h-topology):
 for any h-hypercover  $p:\mathcal X \rightarrow X$
 of $R$-schemes, the functor
$$
p^*:\smod_X \rightarrow \smod_\mathcal X
$$
is fully faithful.

Recall also the following more concrete
 version of descent: given any pseudo-Galois
 cover\footnote{$f$ is finite surjective
 and admits a factorization $f=p f'$ where 
 $f'$ is a Galois cover of group $G$
 and $p$ is radicial.}
 $f:Y \rightarrow X$ of group $G$,
 any syntomic module $M$ over $X$,
 the canonical morphism:
$$
M \rightarrow \big(f_*f^*(M)\big)^G
$$
is an isomorphism, where we have denoted by $?^G$
 the fixed point for the obvious action of $G$.
\end{num}

\bibliographystyle{amsalpha}

\providecommand{\bysame}{\leavevmode\hbox to3em{\hrulefill}\thinspace}
\providecommand{\MR}{\relax\ifhmode\unskip\space\fi MR }
\providecommand{\MRhref}[2]{%
  \href{http://www.ams.org/mathscinet-getitem?mr=#1}{#2}
}
\providecommand{\href}[2]{#2}

\end{document}